\documentclass[12pt]{amsart}
\usepackage{blindtext}
\usepackage{graphicx}
\usepackage{latexsym}
\usepackage[T1]{fontenc}      
\usepackage{amscd}
\usepackage{fancyhdr}
\usepackage{amsmath, amssymb}
\usepackage[ansinew]{inputenc}
\usepackage[all]{xy}
\usepackage{pdflscape}
\usepackage{longtable}
\usepackage{rotating}
\usepackage{verbatim}
\usepackage{hyperref}
\usepackage{subfigure}
\usepackage{mathrsfs}
\usepackage{tensor}
\usepackage{enumerate}
\usepackage{hyperref}
\usepackage[textsize=small]{todonotes}
\usepackage{blindtext}
\usepackage[normalem]{ulem}

\usepackage{color}

\theoremstyle{plain}
\newtheorem{thm}{Theorem}[section]
\newtheorem*{thm*}{Theorem}
\newtheorem{cor}[thm]{Corollary}

\newtheorem*{conjecture*}{Conjecture}
\newtheorem{lem}[thm]{Lemma} 
\newtheorem{prop}[thm]{Proposition}

\theoremstyle{definition}
\newtheorem{defi}[thm]{Definition}
\theoremstyle{remark}
\newtheorem{rem}[thm]{Remark}
\numberwithin{equation}{section}

\newtheorem{example}[thm]{Example}

\newcommand{\AS}{\mathcal{AS}}

\newcommand{\C}{\mathbb{C}}
\newcommand{\N}{\mathbb{N}}
\newcommand{\R}{\mathbb{R}}
\newcommand{\K}{\mathbb{K}}

\newcommand{\Z}{\mathbb{Z}}

\newcommand{\proj}{\mathbb{P}}

\let\mathscr\mathcal

\newcommand{\ps}{\psi}
\newcommand{\ph}{\phi}
\newcommand{\Ps}{\Psi}
\newcommand{\Ph}{\Phi}
\newcommand{\f}{f}

\newcommand{\changes}{}

\newcommand{\aaa}{arc-wise analytic}
\newcommand{\Aaa}{Arc-wise analytic}

\newcommand{\aab}{arc-wise analytically}
\newcommand{\abb}{arc-wise analyticity}

\DeclareMathOperator{\re}{Re{}}
\DeclareMathOperator{\im}{Im{}}
\DeclareMathOperator{\mult}{mult }
\DeclareMathOperator{\grad}{ grad }

\begin{document}
\title
[Arcwise Analytic Stratification]{Arc-wise Analytic Stratification, Whitney Fibering Conjecture and Zariski Equisingularity}
\author{Adam Parusi\'nski and Lauren\c tiu P\u aunescu}
\address {Univ. Nice Sophia Antipolis, CNRS,  LJAD, UMR 7351, 06108 Nice, France}

\email{adam.parusinski@unice.fr}
\address{School of Mathematics, University of Sydney,
  Sydney, NSW, 2006, Australia }%
\email{laurent@maths.usyd.edu.au}%

\thanks{Partially supported  by ANR project STAAVF (ANR-2011 BS01 009), by ARC DP150103113, by Sydney University BSG2014 and by Faculty of Science PSSP}

\begin{abstract}

In this paper we show Whitney's fibering conjecture  in the real and complex, local analytic and global algebraic cases.

For a given germ of complex or real analytic set, we show  the existence of a stratification satisfying  a strong
 (real arc-analytic with respect to all variables and analytic
 with respect to the parameter space) trivialization property along each stratum.
 We call such a trivialization arc-wise analytic and we show that it can be constructed under the classical 
Zariski algebro-geometric equisingularity assumptions.  Using a slightly  stronger version
 of the Zariski equisingularity, we  show the existence of Whitney's stratified fibration,  satisfying the conditions (b) of Whitney and (w) of Verdier.
 Our construction is based on the Puiseux with parameter theorem and a generalization of Whitney's interpolation.  For algebraic sets our construction gives a global stratification.

We also present several applications of the arc-wise analytic trivialization, mainly to the stratification theory and the equisingularity of analytic set and function germs.
 In the real algebraic case, for an algebraic family of projective varieties, we show that  the
 Zariski equisingularity implies local constancy of the associated weight filtration.

 \end{abstract}


\maketitle

\tableofcontents


{\Large \part*{Introduction and statement of results.}}
 \medskip

In 1965 Whitney stated the following conjecture.  
\medskip

\begin{conjecture*}\label{conjecture}
 {\rm [Whitney fibering conjecture, \cite{whitney} section 9, p.230]}
Any analytic subvariety $V\subset U$ ($U$ open in $\C^n$) 
has a stratification such that each point $p_0\in V$ has a neighborhood $U_0$ with a semi-analytic fibration. 
\end{conjecture*}

By a semi-analytic fibration Whitney meant the following (it has nothing to do with the notion of semi-analytic set introduced about the same time by \L ojasiewicz in \cite{lojasiewicz}).   Let $p_0$ belong to a stratum $M$ and let 
$M_0=M\cap U_0$.  Let $N$ be the analytic plane orthogonal to $M$ at $p_0$ and let 
$N_0=N\cap U_0$.  Then Whitney requires that there exist a homeomorphism  
$$\phi (p,q) : M_0\times N_0 \to U_0,$$ 
 complex analytic in $p$, such that $\phi (p,p_0)= p$ 
($p\in M_0$) and $\phi (p_0,q) =q$ ($q\in N_0$), and preserving the strata.  
He also assumes that for each $q\in N_0$ fixed, $\phi (\cdot,q) : M_0\to U_0$ is a complex analytic embedding onto an analytic submanifold $L(q)$ called the fiber (or the leaf) at $q$, and thus $U_0$ is fibered continuously into submanifolds complex analytically diffeomorphic to $M_0$.   
Note that due to the existence of continuous moduli it is in general impossible to find $\phi (p,q)$ complex analytic in both variables, see \cite{whitney}.  

Whitney stated his conjecture in the context of his regularity conditions (a) and (b) for  stratifications 
introduced in  \cite{whitneyannals}.   These  conditions imply the topological triviality (equisingularity) along each stratum.  This trivialization is obtained by the flow  of some  "controlled" vector fields and does not imply the existence of a  fibration as required in Whitney's conjecture.  Thus Whitney conjectured the existence of a better trivialization, given by his fibration, that should, 
moreover, imply  the regularity conditions (a) and (b).  As Whitney claims in \cite{whitney} a semi-analytic (in his sense) fibration ensures the continuity of the tangent spaces to the leaves of the fibration and hence Whitney's condition (a) for the stratification.  This seems not to be  obvious.  We recall Whitney's argument in Subsection \ref{remark}, but to complete it we need an extra assumption.  To have the condition (b),  quoting Whitney,  
"one should probably require more than just the continuity of $\phi$ in the second variable".    

 Whitney's fibering conjecture as stated above  was proven by Hardt and Sullivan in the local analytic and global projective cases,  in Theorem 6.1 of \cite{hardtsullivan}.   But it is not clear to us whether $\phi$ of \cite{hardtsullivan} ensures the continuity of the tangent spaces   to the leaves or  the condition (b).  In the real algebraic case an analog of 
 Whitney's conjecture was proven in \cite{hardt1983}.  In this case  the continuity of the tangent 
 spaces is not clear either.  
 
 Whitney's fibering conjecture  has  been studied in the context of abstract $C^\infty$ stratified spaces and 
 topological equisingularity, cf. \cite{murolotrotman1999}, \cite{murolotrotman2000}, \cite{murolotrotman2006}. 
Assuming that the conjecture is true, Murolo and Trotman have shown in \cite{murolotrotman2006} a horizontally-$C^1$ version of   Thom's first isotopy theorem.

\subsection{Ehresmann Theorem}  Whitney's conjecture is consistent with the following holomorphic version of the Ehresmann fibration theorem, see \cite{voisin}.  Let $\pi: \mathcal X\to B$ be a proper holomorphic 
submersion of complex analytic manifolds.  Then,  for every $b_0\in B$ there is a neighborhood $B_0$ of $b_0$ in $B$ and a 
$C^\infty $ trivialization $$\phi (p,q) : B_0\times X_0 \to \mathcal X_{B_0},$$ 
 holomorphic in $p$, where $X_0=\pi ^{-1} (b_0)$, $\mathcal X_{B_0}=\pi ^{-1} (B_0)$.  
Note that $\phi$ can be made real analytic but, in general, due to the presence of continuous moduli, not  holomorphic.  This version of Ehresmann's theorem is convenient to study the variation of Hodge 
structures in families of K\"ahler manifolds, see  \cite{voisin}.

As we show in this paper there are no continuous moduli for complex analytic families of singular 
complex analytic germs, nor for families of algebraic varieties, 
provided  $\varphi$ is assumed complex analytic in $p$ and real arc-analytic in $q$, see Theorems 
 \ref{theoremconjecture} and \ref{theoremequisingularity3} and Lemma  \ref{change}.


\subsection{Statement of main results}
In this paper we show Whitney's  fibering conjecture  in the real and complex,   local analytic and global algebraic cases.  
For this, for a given germ of complex or real analytic set, we show  the existence of a stratification 
that can locally  be trivialized by a map $\ph(p,q)$ that is 
not only real/complex analytic  (depending on the case)  in $p$, continuous in both variables, but also  {\aaa},  see Definition \ref{maindefinition2}. In particular,  both $\ph$ and $\ph^{-1}$ are analytic on real analytic arcs.  Moreover, for every real analytic arc $q(s)$ 
in $N_0$,  $(p,s) \to \phi (p,q(s))$  is  analytic.  As we show in Proposition \ref{continuitytangent} 
this ensures the continuity of tangent spaces to the fibers and hence Whitney's condition (a) on the stratification (both in the real and complex cases).  Then, by additionally requiring that the trivialization 
preserve the size of the distance to the stratum, we  show the existence of Whitney's fibration satisfying the conditions (b) of Whitney and (w) of Verdier \cite{verdier}.   We call such an   {\aaa}    trivialization  regular along the stratum, Definition \ref{maindefinition3}.

\begin{thm*}[Theorem \ref{theoremconjecture}]
Let $\mathcal X = \{X_i\}$ be a finite family of analytic subsets of an open $U\subset \K^N$, ($\K$ denotes $\R$ or $\C$).  Let 
$p_0\in U$.  Then there exist an open neighborhood $U'$ of $p_0$ and an analytic stratification of 
 $U'$  compatible with each $U'\cap X_i$ admitting regular {\aaa } trivialization along each stratum.   
\end{thm*}

In Section \ref{functions} we extend these results to stratifications of analytic functions.  
Recall that a stratification of a $\K$-analytic function  $f:X\to \K$  is a  stratification of $X$ such that the zero set $V(f)$ of $f$ is a union of strata.  Theorem  \ref{a_f} together with Proposition 
\ref{complexregularity} implies the following result.  

\begin{thm*}[$\K=\C$]
If a stratification of $f$ admits  an {\aaa } trivialization along a stratum $S\subset V(f)$ then it satisfies  the Thom condition ($a_f$) 
along this stratum.  If such trivialization is, moreover, regular along $S,$ then it satisfies   the strict Thom condition ($w_f$) 
along $S$. 
\end{thm*}

We also give  an analogous  result in the real case using the notion of  regularity of a function for 
an {\aaa } trivialization, defined in Subsection \ref{regularfunctions}.  
Thom's  regularity conditions are used to show topological triviality of functions along strata.  We discuss this in detail  
in Section \ref{functions}, where we develop three different constructions guaranteeing such triviality.

In Section \ref{algebraiccase} we treat the algebraic case.  By reduction to the homogeneous analytic case we show the following results. 

\begin{thm*}[Theorem \ref{algebraictheoremconjecture}]  
Let $ \{V_i\}$ be a finite family of algebraic subsets of  $\proj _\K^n$.    Then there exists an 
algebraic stratification of $\proj _\K^n$  compatible with each $ V_i$ and admitting semialgebraic regular {\aaa } trivializations along each stratum.  
\end{thm*}

\begin{thm*}[Theorem \ref{theoremequisingularity3}]  
Let $T$ be an algebraic variety  and let $\mathcal X =\{X_k\}$ be  a finite family of 
algebraic subsets $T \times \proj^{n-1}_\K$.  Then there exists an algebraic stratification  $\mathcal S$ of $T$ such that for every stratum $S$ and for 
every $t_0\in S$  there is a neighborhood $U$ of $t_0$ in $S$ and a semialgebraic, {\aaa } 
 trivialization  of $\pi$,   preserving the family $\mathcal X$ 
\begin{align}
\Ph  : U \times \proj^{n-1}_\K \to \pi^{-1} (U),\end{align}
$\Ph(t_0,x)= (t_0,x)$, 
where $\pi:T \times \proj^{n-1}_\K\to T$  denotes the projection.   
\end{thm*}

The {\aaa } triviality is particularly friendly to the curve selection lemma argument.  
Recall that in analytic geometry many properties can be proven by checking them  along real analytic arcs.  
We use this argument many times in this paper.  For precise statement and a proof of the curve selection lemma 
we  refer the reader to  \cite{bruhatcartan}, \cite{wallace}, \cite{lojasiewicz}, \cite{Hironaka73} or \cite{milnor}.  
To prove the classical regularity conditions, (a) of Whitney, (w) of Verdier, or Thom's conditions (a$_f$) or 
(w$_f$), we use a wing lemma type argument originated by Whitney in \cite{whitneyannals}, 
see Proposition \ref{winglemma}.   {\Aaa } trivializations naturally provide such wings.  
For instance, in Whitney's notation, if $q(s)$ is a real analytic arc in $N_0$ then $\phi (p,q(s))$  
constitutes such an {\aaa } wing.  Moreover, {\aaa } trivializations preserve the multiplicities and 
the singular loci of the sets they trivialize, see Propositions \ref{regularityequimultplicity} and  \ref{singularities} for precise statements.

Thus this paper,  in order to get local  {\aaa} trivializations, we redefine  many classical notions and reprove many classical results of stratification theory on analytic and algebraic sets.  
  Our approach is based on the classical Puiseux with parameter theorem and the  algebro-geometric equisingularity of Zariski (called also Zariski's equisingularity).  
  Our main tool in the construction of   {\aaa } trivializations is Theorem \ref{theorem}, which says that  the Zariski equisingularity implies {\aaa } triviality.   To show it, we use Whitney's interpolation   adapted to arc-analytic geometry.  This is explained in  Appendix \ref{Sec:interpolation}. 

Besides the proofs of the Puiseux with parameter theorem and the curve selection lemma this paper is self-contained.  Our method is based on the Zariski equisingularity, hence is constructive; it involves the computation of the discriminants of subsequent linear projections.


\subsection{Zariski Equisingularity}

Let $V$ be a real or  complex analytic variety.  Then there exists a stratification $\mathcal S$ of $V$ such that $V$ is equisingular  along each stratum  $S$.   There are several different notions of equisingularity, the basic  one 
is the topological one, with many possible refinements, such as stratified 
topological triviality.  
Whitney introduced in \cite{whitney}, \cite{whitneyannals}, the regularity conditions (a) and (b) that guarantee, by the Thom-Mather first isotopy theorem, the topological equisingularity along each stratum.  He showed in \cite{whitneyannals} that any complex analytic variety admits  (a) and (b) regular stratifications.  
The real analytic case was established in \cite{lojasiewicz} and the subanalytic case in \cite{Hironaka73}.

Topological equisingularity can also  be obtained by means of  the Zariski equisingularity, as shown by Varchenko in 
\cite{varchenko1972, varchenko1973, varchenkoICM}.   
Zariski's definition, see \cite{zariski}, is recursive and is based on the geometry of discriminants.   
Let $V\subset \K^N$ be a  hypersurface.  We say that $V$ is Zariski equisingular along stratum 
$S$ at $p\in S$ if, after a change of a local system of coordinates, the discriminant of a linear projection $\pi:\K^N \to \K^{N-1}$ restricted to $V$ is equisingular 
along $\pi(S)$ at $\pi(p)$.  The kernel of $\pi$ should be transverse to $S$ and $\pi$ restricted to $V$ 
should be finite at $p$.  Stronger notions of Zariski's equisingularity are obtained  if one assumes that the kernel of $\pi$ is not contained in the tangent cone to $V$ at $p$ (transverse Zariski equisingularity) or that $\pi$ is generic (generic Zariski equisingularity).  

The special case, when  $S$ is of codimension one in $V$, was studied by Zariski in \cite{zariskiI}.  
Note that in this case $V$ can be considered as a family of plane curves parameterized by $S$.  
As Zariski shows, in this case the Zariski equisingularity is equivalent to Whitney's conditions (a) and (b) on the pair of strata 
$V\setminus S,S$.  Such equisingular families of plane curves  admit a uniform Puiseux representation parameterized by $S$, this 
result is  known  in literature as  the parametrized Puiseux or the Puiseux with parameter theorem.  We recall it in Subsection \ref{puiseux}.  

In this paper we show that the Zariski equisingularity implies {\aaa } triviality.   In the case of the
Zariski transverse equisingularity we obtain an  {\aaa } triviality that is also regular.  

\begin{thm*}[see Theorems \ref{theorem} and \ref{theoremtransverse}]
If a hypersurface $V\subset \K^N$ is Zariski equisingular along stratum $S$ at $p\in S$, then 
there is a local {\aaa} trivialization of $ \K^N$ along $S$ at $p$ that preserves $V$.  
\end{thm*}

Our proof is different from that of Varchenko and is based on Whitney's interpolation that gives 
a precise algebraic formula for such a trivialization.   The main idea is  the
following.  Suppose $V$ is Zariski equisingular along $S$ and $\pi:\K^N \to \K^{N-1}$ is 
the projection giving this equisingularity.  By the inductive assumption, there is an  {\aaa } trivialization of $\pi (V)$  along $\pi(S)$.  This trivialization is then lifted to a trivialization of $V$ along $S$, 
and  extended to a trivialization of the ambient space $\K^n$ along $S$  by  our version of Whitney's interpolation.   Therefore the lift  is continuous, subanalytic, and, by  the Puiseux with parameter theorem,  {\aaa }.  
 This latter conclusion is obtained thanks to the {\abb } in  the inductive assumption, 
 see Remark \ref{recursiveargument}.  
 
For an analytic function germ $F$  we denote by $F_{red}$ its reduced (i.e. square free) form.  Let $(Y,y)$ be a germ of a 
$\K$-analytic space.  For a monic polynomial $F\in \mathcal O_Y [z]$  in $z$ we often consider the discriminant of $F_{red}$.  
If  $Y$ has arbitrary singularities then this discriminant should be replaced by an appropriate 
generalized discriminant that is a polynomial in the coefficients of $F$, see Appendix \ref{Part:discriminants}. 

 Finally, we note that the Zariski equisingularity can be used to trivialize not only hypersurfaces but 
 also analytic spaces of arbitrary embedding codimension.  This follows from the fact that if a hypersurface $V$ is Zariski equisingular along $S$ and  $V=\cup V_i$ is the decomposition of $V$ into irreducible components,  then the {\aaa } trivialization preserves 
 each $V_i$ and hence any set-theoretic combination of the $V_i$'s.  
 
\subsection{Proofs of the main theorems are constructive.}  
The main theorems, Theorem \ref{theoremconjecture} and Theorem 
\ref{algebraictheoremconjecture} can be shown in a virtually algorithmic way.  
For this we proceed as follows.  Given an ideal $\mathcal I$ of $\K[x_1, . . . ,x_n]$ or 
$\K\{x_1, . . . ,x_n\}$ we choose a finite set of generators of $\mathcal I$ and consider their product 
$f(x_1, . . . ,x_n)$.  Then we complete $f$ to a system of  (pseudo)polynomials $F_i(x_1, . . . ,x_i)$, 
$i=1,. . . ,n$, see Definition \ref{systempseudopolynomials}.  
This process, explained in detail in subsections \ref {proof} and \ref{construction}, 
involves a generic linear change of coordinates.  That is the only point not entirely algorithmic.  
It follows from Theorem \ref{theorem}, see Proposition 
\ref{canonicalstratification},  that the canonical stratification associated to a system of (pseudo)polynomials admits locally {\aaa} trivializations.   
To get a regular arc-wise analytic stratification, and hence a Whitney stratification,  we need to refine this construction and  consider not only $f$ but also its partial derivatives with respect to $x_n$.  This way we get a system of  (pseudo)polynomials $F_i$ 
that we call derivation complete, as explained in  Example \ref{derivationcomplete}.

\subsection{Zariski Equisingularity and regularity conditions on stratifications}

In general,  Whitney's conditions   and 
 Zariski's  equisingularity, do not imply one another.  We recall several classical examples in Section \ref{examples}.    By Zariski \cite{zariskiI},  they coincide for a hypersurface $V$ along a nonsingular subvariety of codimension $1$ in $V$.  

It was shown by Speder  \cite{speder} that in the complex case Zariski's equisingularity obtained 
by taking generic projections implies the regularity conditions (a) and (b) of Whitney.  
As it follows from our Theorem \ref{theoremtransverse} the assumption that the projections are transverse, 
 in both 
complex and real cases,  is sufficient.  
We also show in Proposition \ref{equimultiplicity} that the Zariski equisingularity (arbitrary projections) implies equimultiplicity.   

Whitney's stratification approach is independent of the choice of local analytic coordinates and simple to define.  
But the trivializations obtained by this method are not explicit and difficult to handle.  These trivializations are obtained by 
integration of "controlled" vector fields whose existence can be theoretically established.  Stronger regularity conditions, 
such as (w) of Verdier \cite{verdier}, or Lipschitz of Mostowski \cite{mostowski}, \cite{parusinskilipschitzreview}, 
lead to easier constructions of such vector fields, but in general, even if these vector fields can be chosen subanalytic,
 not much can be said about their flows. 

Zariski's  equisingularity method is more explicit and in a way constructive.  It uses the actual equations 
and local coordinate systems.  This can be considered either as a drawback or as an advantage.   
Zariski's equisingularity was used, for instance,  by Mostowski \cite{mostowski2}, see also 
\cite{BPR}, to show that analytic set germs are always homeomorphic to algebraic ones.  

In this paper we apply the Zariski equisingularity to construct stratifications via  corank one linear projections.  
This method was developed by Hardt and Hardt \& Sullivan 
\cite{hardt1975,hardt1980,hardt1983, hardtsullivan}.  

In  \cite{zariski1979} Zariski has proposed a general theory of equisingularity for hypersurfaces  by introducing the notion of dimensionality type of their points.  The dimensionality type is defined through an inductive process, using discriminants of generic (not necessarily linear) projections.  Besides the codimension one case  \cite{zariskiI},  this notion has been studied in the codimension two for families of isolated 
surface singularities in \cite{brianconhenry} and \cite {teissier1982}. 


\subsection{Applications to real algebraic geometry} 
Semialgebraic arc-analytic maps are often used in real algebraic geometry.  
The arc-analytic maps were introduced by Kurdyka in \cite{kurdyka1}.  It was shown by Bierstone and Milman in \cite{B-M} (see also  \cite{parusinski1994}) that  semialgebraic arc-analytic maps 
are blow-analytic.  Semialgebraic arc-analytic maps and 
semialgebraic arc-symmetric sets were used in \cite{kurdyka1999}, \cite{parusinski2004},  to show that injective self-morphisms of real algebraic varieties are surjective.    
For more on this development we refer the reader to \cite{kurdykaparusinski}.   
Let us also note that recently studied \cite{kollarnowak}, \cite{kucharz2009}, \cite{kucharz2014}, \cite{FHMM}  continuous rational maps are, in particular, arc-analytic and 
semialgebraic.

The weight filtration on real algebraic varieties, recently introduced  \cite{mccroryparusinski1,mccroryparusinski2}, is stable 
under 
semialgebraic arc-analytic homeomorphisms.  By Theorem \ref{theoremequisingularity3}
any algebraic family of algebraic sets is generically semialgebraically {\aaa } trivial, 
and therefore we have the following result.

\begin{thm*}[see Corollary \ref{corollaryweight}] 
Let $T$ be a real algebraic variety and let $X$ be  an algebraic subset of  
$T \times \proj^{n-1}_\K$.  Then there exists a 
 finite stratification  $\mathcal S$ of $T$ such that for every stratum $S$ and for 
every $t_0, t_1 \in S$  the fibers  $X_{t_0}$ and $X_{t_1}$ have isomorphic  the weight filtration on homology.   
\end{thm*}


\subsection{Resolution of singularities and blow-analytic equivalence} 

The resolution of singularities can also be used to show topological equisingularity,  though the results are partial and many  questions are still open.  This method works for the families of isolated singularities, cf.  Kuo \cite{kuo1985}, and gives local arc-analytic trivializations.  
But little is known if the singularities are not isolated, see e.g. \cite{koike}.   
Let us explain the encountered problem on a simple example.  
 Suppose that  $Y\subset V$ is nonsingular and let $\sigma:\tilde V\to V$ be a resolution of singularities such that $\sigma ^{-1} (Y)$ is a union of the components of exceptional divisors.   Fix a 
 local projection  $\pi:V\to Y$.   The exceptional divisor of $\sigma$ as a divisor with normal crossings is naturally stratified by the intersections of its components.  Let $Z\subset Y$ be the closure of the union of all critical values of $\pi \circ \sigma$ restricted to the strata.  By Sard's theorem $\dim Z < \dim Y$.  
 We say that $V$ is \emph{equiresoluble along} $Y$ if $Y\cap Z= \emptyset$.  Thus $V$ is equiresoluble along $Y'=Y\setminus Z$ and   $\pi \circ \sigma$ is locally topologically (and even real analytically) trivial over $Y'$.  If $\sigma$ is an isomorphism over $V\setminus Y$ (family of isolated singularities case) 
 then this trivialization blows down to a topological trivialization of a neighborhood of $Y$ in $V$.  But  
 in the non-isolated singularity case there is no clear reason why  a trivialization  of $\pi \circ \sigma$  comes from a topological trivialization of a neighborhood of $Y$ in $V$.   
  Thus, in general, we do not know whether equiresolubility implies topological equisingularity.  
  
As before, one may ask how the equiresolution method is related to the other methods of establishing topological 
equisingularity.   A  non-trivial result of Villamayor \cite{villamayor}, says  that the generic Zariski equisingularity of a hypersurface implies a weak version of equiresolution, see loc. cit. for details, but  
the main problem remains, it does not 
show the existence of a topological trivialization that lifts to the resolution space.  

\subsection*{Notation and terminology.} 
We denote by $\K$ either  $\R$ or $\C$.  Thus, by $\K$-analytic we mean either real analytic or holomorphic (complex analytic).

By an analytic space we mean one in the sense of \cite{narasimhan}. As we work only locally in the analytic case, it suffices to consider only analytic set germs.  For an analytic space $X$ by $Sing (X)$ we denote the set of singular points of $X$ and by $Reg(X)$ its complement, the 
set of regular points of $X$.   For an analytic function germ $F$  we denote by  
$V(F)$ its zero set and by $F_{red}$ its reduced (i.e. square free) form.  By a real analytic arc we mean  a real analytic map $\gamma : I\to X$, where $I=(-1,1)$ and $X$ is a real or a complex analytic space.

\subsection*{Acknowledgements.} 
We would like to thank Goulwen Fichou, Tzee-Char Kuo, and David Trotman for encouragement and  several remarks and suggestions concerning the paper.

  

{\Large \part {{\Aaa } trivializations.}\label{part1}}
\medskip


\section{Definition and basic properties} \label{Aaa}


Let $Z, Y$ be  $\K$-analytic spaces.  A map $f : Z\to Y$ is called {\it arc-analytic} if $f\circ \delta$  is analytic for every real analytic arc $\delta : I \to Z$,  where $I=(-1,1)\subset \R$.  The arc-analytic maps were introduced by 
Kurdyka in \cite{kurdyka1} and have been subsequently  used intensively in real analytic and algebraic geometry, 
 see  \cite{kurdykaparusinski}.   It was shown by Bierstone and Milman in \cite{B-M}
  (see also \cite{parusinski1994} for a different proof) that the   arc-analytic maps with subanalytic graphs are continuous and that the arc-analytic maps with semi-algebraic graphs are blow-analytic, i.e. can be made real analytic after composing with blowings-up.    Therefore the arc-analytic maps are closely related  to the blow-analytic trivialization in the sense of Kuo \cite{kuo1985}.  
  
In this paper we consider arc-analytic trivializations satisfying some additional properties.  
 Below we  define  the notion of  arc-wise analytic trivialization, that is not only arc-analytic with  arc-analytic inverse, but it is also $\K$-analytic with respect to the parameter $t\in T$.  
For simplicity we assume that the parameter space $T$ is nonsingular.

\begin{defi}\label{maindefinition1}
Let $T, Y, Z$ be  $\K$-analytic spaces, $T$ nonsingular.  We say that a map 
$f (t,z) : T \times Z\to Y$ is 
\emph{{\aaa } in} $t$ 
if it is $\K$-analytic in $t$ and arc-analytic in $z$, that is if for every real analytic arc  $z (s) : I\to Z$,  the map 
$ f(t, z(s))$ is real analytic, and moreover, if $\K=\C$, complex analytic with respect to $t$.   
\end{defi}


All {\aaa } maps considered in this paper are  subanalytic and hence continuous.

We stress that even for complex analytic spaces we define the notion of arc-analyticity using only 
real analytic arcs. (A map of complex analytic spaces $f : Z\to Y$, with $Z$ nonsingular, that is complex analytic on complex analytic arcs is, by Hartogs Theorem, complex analytic.)

\begin{defi}\label{maindefinition2}
Let $ Y, Z$ be  $\K$-analytic spaces and let $T$ be a nonsingular $\K$-analytic space.  
Let $\pi : Y\to T$ be a $\K$-analytic   map.  
 We say  
 $$\Ph (t,z) : T \times Z\to Y$$
  is an  \emph{{\aaa} trivialization of $\pi$} if it satisfies the following properties 
  \begin{enumerate}
  \item 
  $\Ph$  is a subanalytic homeomorphism,   
   \item 
  $\Ph$ is {\aaa} in $t$ (in particular it is $\K$-analytic with respect to $t$), 
  \item
   $\pi \circ \Ph (t,z)=t$  for every  $(t,z) \in T\times Z$, 
     \item
   the inverse of $\Ph$ is arc-analytic, 
    \item 
  there exist $\K$-analytic stratifications $\{Z_i\}$ of $Z$ and $\{Y_i\}$ of $Y,$ such that for each $i$,   
  $Y_i = \Ph (T\times Z_i)$  and $\Ph_{|T\times Z_i} : T\times Z_i \to Y_i$ 
  is a real analytic diffeomorphism.  
   \end{enumerate}
  Sometimes we say for short that such $\Ph$ is an {\aaa } trivialization if it is obvious from the context 
what the projection $\pi$ is.  
\end{defi}

In the algebraic case we require $\Ph$ to be semialgebraic and that the stratifications  
are algebraic in the sense  explained in Section \ref{stratifications}.

If $\Ph(t,z) : T \times Z\to Y$ is an  \emph{{\aaa} trivialization} then, for each $z\in Z$, the map 
$T\ni t \to \Ph(t,z)\in Y$ is a $\K$-analytic embedding.  We denote by   $L_z$ its image and we call 
it \emph{a leaf or a fiber of $\Ph$.}  We say that $\Ph$ \emph{preserves $X\subset Y$} if $X$ is a union of 
leaves.  We denote by 
$T_y= T_{y} L_z$, $y= \Ph (t,z)$,  the tangent space to the leaf through $y$.  


\subsection{Computation in local coordinates}\label{coordinates}

Let $(t_0,z_0) \in T\times Z$, $y_0=\Ph (t_0,z_0)$.   Choosing local coordinates, we may always  
assume that $(T,t_0)=(\K^m,0)$, $(Z,z_0)$ is an analytic subspace of $(\K^n,0)$, and $(Y,y_0)$ is an analytic subspace of $(T\times \K^n,0)$ with  $\pi(t,x) =t$.  Thus we may write 
\begin{align}\label{form}
\Ph (t,z) = (t, \Ps(t,z)).
\end{align} 
We also suppose that $L_0 = \Ph(T\times \{0\}) = T\times \{0\}$ as germs at the origin. 

Using local coordinates, we  identify $T_y$ with an $m$-dimensional vector subspace of $\K ^m \times \K^n$ and consider $T_y$ as a point in the Grassmannian $G(m,m+n)$.  
These tangent spaces are spanned by the vector fields $v_i$ on $Y$ defined by 
\begin{align}\label{vectorfields}
v_i (\Ph(t,z)) := (\partial /\partial t_i, \partial \Ps /\partial t_i) \qquad i=1,...,m.
\end{align}

\begin{prop}\label{continuitytangent}
 Let  $\Ph (t,z) : T \times Z\to Y$ be an  {{\aaa} trivialization}.  Then the vector fields $v_i$ and the 
 tangent space map $y\to T_y$ are subanalytic, arc-analytic,  and continuous.  
\end{prop}

\begin{proof}
The subanalyticity follows from the classical argument of subanalyticity of the derivative of a subanalytic map, see \cite{kurdykasmooth} Th\'eor\`eme 2.4.   Let $(t(s),z(s)): (I,0) \to (T\times Z,(t_0,z_0))$ be a 
real analytic arc germ. Consider the map $\Ps : T\times I \to \K^n$
\begin{align}\label{development}
\Ps(t, z(s)) = \sum _{k\ge k_0} D_k(t) s^k . 
\end{align}
The arc-analyticity of $v_i$ on $(t(s),z(s))$ follows  from the analyticity of 
$(t,s) \to \partial \Ps(t, z(s))  /\partial t_i$.  
Finally, subanalytic and arc-analytic maps are continuous, cf. \cite{B-M} Lemma 6.8.
\end{proof}

\begin{rem}\label{uniquesolution}
For  $y = \Ph (t,z)$ fixed,   $\tau \to \Ph (t+\tau e_i, z) $ is an integral curve of $v_i$ through $y$. 
 Moreover, such an integral curve  is unique as follows  from (5) of Definition \ref {maindefinition2}.  
\end{rem}


\subsection{{\Aaa}  trivializations regular along a fiber}\label{regular}

We now define regular {\aaa}  trivializations along a fiber that will be important for applications 
in stratification theory including our proof of Whitney's fibering conjecture, c.f. section \ref{stratifications}.  
Regular  {\aaa}  trivializations  preserve the size of the distance to a fixed fiber.  

\begin{defi}\label{maindefinition3}
We say that an {\aaa} trivialization  $\Ph(t,z) : T \times Z\to Y$  is \emph{regular at} $(t_0,z_0) \in  T \times Z$ if  there are a neighborhood $U$ of  $(t_0,z_0)$ and a constant $C>0$  such that   for all $(t,z) \in U$  
(in local coordinates at $(t_0,z_0)$ and $y_0 = \Ph(t_0,z_0)$) 
 \begin{align}\label{regularitybound}
  C^{-1} \|\Ps (t_0,z)\| \le \|\Ps (t,z))\| \le  C \|\Ps (t_0,z)\|  ,
\end{align}
where as in \eqref{form},  $\Ph (t,z) = (t, \Ps(t,z))$, $\Ps(t,z_0) \equiv 0$.  
We say that $\Ph$ is \emph{regular along} $L_{z_0}$ if it is regular at every $(t,z_0)$, $t\in T$.  
\end{defi}

 We have the following criterion of regularity which follows from the more general  Proposition \ref{criterionregularfunction} that we prove in the next subsection.  

\begin{prop}\label{criterionregular}
The {\aaa } trivialization $\Ph (t,z)$ is regular at $(0,0)$ if and only if for every real analytic arc germ 
$z(s): (I,0) \to (Z,0)$,  
the leading coefficient of  \eqref{development}  does not vanish at $t=0$:  $D_{k_0}(0) \ne 0$.  

Moreover, if $\Ph(t,z)$ is regular at $(0,0),$ then in a neighborhood of $(0,0) \in T\times Z$ 
\begin{align}\label{boundedder}
 \| \frac {\partial \Ps} {\partial t} (t,z))\| \le  C \|\Ps (t,z)\| . 
\end{align}
\end{prop}

\subsection{Functions and maps regular along a fiber}\label{regularfunctions}  
In this section we generalize the notion of regularity for {\aaa} trivializations to $\K$-analytic function germs  
$f: (Y,y_0) \to (\K, 0)$, see Definition \ref{regularfunction}.  First we show the following criterion that we state for $f$ of a slightly more general form.

\begin{prop}\label{criterionregularfunction}
Let $\Ph(t,z) : T \times Z\to Y$ be  an {\aaa} trivialization and let   $f: (Y,y_0) \to (\R^ k, 0)$, $y_0=\Ph(t_0,z_0)$, be a real analytic map germ.  Then the following conditions are equivalent:
\begin{enumerate}
\item[(i)]
there is $C>0$ such that  for all $(t,z)$ sufficiently close to $(t_0,z_0)$ 
\begin{align}\label{regularfunctioncondition}
  C^{-1} \|f(\Ph (t_0,z))\| \le \|f(\Ph (t,z))\| \le  C \|f(\Ph (t_0,z)) \|   .
\end{align}
\item[(ii)]
for every real analytic arc germ $z(s): (I,0) \to (Z,z_0)$  
the leading coefficient $D_{k_0}$ of  
\begin{align}\label{developmentfunction}
f(\Phi (t, z(s))) = \sum _{k\ge k_0} D_k(t) s^k 
\end{align}
satisfies $D_{k_0}(t_0) \ne 0$.  
\item[(iii)]
there is $C>0$ such that  for all $(t,z)$ sufficiently close to $(t_0,z_0)$
\begin{align}\label{boundderfunction}
 \| \frac {\partial (f\circ \Phi)} {\partial t} (t,z))\| \le  C \| f\circ \Phi  (t,z)\| .  \qquad \qquad 
\end{align}  
\end{enumerate}
\end{prop}

\begin{proof}
To show that (ii) implies (i) we use  the curve selection lemma.  If (i) fails then there is a real analytic arc germ $(t(s),z(s)): (I,0) \to (T\times Z,(t_0,z_0))$ along which one of the inequalities of (i) fails, that is, for instance, 
$\frac {\|f(\Ph (t(s),z(s)))\|} { \|f(\Ph (t_0,z(s))) \|} \to \infty$ as $s\to 0$.  But this contradicts (ii).  To complete this argument we note that $f(\Ph (t_0,z(s))) \equiv 0$ iff $f(\Ph (t,z(s)))\equiv 0$, that is what (ii) means in this case.  
Clearly (ii) follows from (i).  

Similarly, it is sufficient  to show (iii) on every real analytic arc and this follows immediately from 
(ii).  Finally (i) follows from (iii).
\end{proof}


\begin{defi}\label{regularfunction}
Let $\Ph(t,z) : T \times Z\to Y$ be  an {\aaa} trivialization in $t$.  We say that an analytic function germ 
$f: (Y,y_0) \to (\K, 0)$, $y_0=\Ph(t_0,z_0)$,   is  $\Ph$-\emph{regular} (regular for short),  if  it satisfies one of the equivalent conditions of Proposition \ref{criterionregularfunction}.  

We say that $f$ is $\Ph$-\emph{regular along} $L_{z_0}$ (regular for short) if it is regular at every $(t,z_0)$, $t\in T$.  
\end{defi}




\begin{prop}\label{regularproducts}
Let $\Ph(t,z) : T \times Z\to Y$ be  an {\aaa} trivialization and let $f,g: (Y,y_0) \to (\K, 0)$ be two analytic function germs 
not vanishing identically on each component of $(Y,y_0)$.   Then $f$ and $g$ are regular if and only if  so is  $fg$.  
\end{prop}

\begin{proof}
It  follows from (ii) of Proposition \ref{criterionregularfunction}.  
\end{proof}

In the complex case the regularity is a geometric notion as the following proposition shows.  

\begin{prop}\label{complexregularity}
Suppose $\K=\C$.  
Let $\Ph(t,z) : T \times Z\to Y$ be  an {\aaa} trivialization and let $f: Y\to \C$ be a 
complex analytic function.  Suppose that $\Ph$ preserves $V(f)$.   Then 
$f$ is $\Ph$-regular  at every point of $V(f)$. 
\end{prop}

\begin{proof}
Suppose that this is not the case.  Then there exists a real analytic arc $z(s): (I,0) \to (Z,z_0)$, 
such that in \eqref{developmentfunction},  $D_{k_0}\not \equiv 0$ and $D_{k_0} (t_0)=0$.  
Clearly  $f\circ \Phi(t_0, z(s)) \ne 0$ for $s\ne 0$.  
We show that for $s\ne 0$ there is $t(s)$, $t(s) \to t_0$ as $s\to 0$, such that $f\circ \Phi(t(s), z(s)) =0$.  This 
would contradict  the assumption on $\Ph$.  For this, by restricting to a $\K$-analytic arc  through $t_0$, 
we may suppose that $t$ is a single variable $t\in (\C,0)$.  Let us then write 
$f\circ \Phi (t, z(s)) =s^{k _0} h(t,s)$, where 
\begin{align}\notag
h(t,s) = D_{k_0} (t) +  \sum _{k> k_0} D_k(t) s^{k-k_0}  .
\end{align} Since $0$ is an isolated root of $h(t,0)=0$,  Rouch\'e's Theorem implies that $h(t,s)=0$ 
has roots also for $s\ne 0$. 
\end{proof}

\begin{defi} 
We say that an ideal $\mathcal I$ of $\mathcal O_{Y, y_0}$ is \emph{$\Ph$-regular} (regular for short) if, 
for one or equivalently for every finite system of generators $f_1, ... , f_k$ of $\mathcal I$, 
$f=(f_1, \ldots, f_k)$ satisfies the  equivalent conditions of Proposition  \ref{criterionregularfunction}.   
\end{defi}

This definition generalizes both the notion of regularity of a function and of  a fiber.  It follows that $\Ph$-regularity of an ideal implies that its zero set $V(\mathcal I)$ is preserved by the trivialization,.  But except the complex function case, Proposition \ref{complexregularity}, this is a strictly stronger condition.  We will need the following lemma for the proof  of Proposition 
\ref {winglemma}.   

\begin{lem}\label{complexification}
Suppose $\K=\C$.  Let $\mathcal I$   be an ideal of $\mathcal O_{Y,y_0}$ and let $\Ph(t,z) : T \times Z\to Y$ be  an {\aaa} trivialization.  Let $z(s): (I,0) \to (Z,z_0)$ be real analytic and denote by 
$\varphi (t,s):(T\times \C, t_0\times 0) \to (Y,y_0) $ the complexification of $\Ph (t,x(s))$.  
Suppose that $\varphi (t,s) \notin V(\mathcal I)$ for $s\in \R, s>0$.   Then 
$\varphi (t,s) \notin V(\mathcal I)$ for $s\in \C\setminus \{0\}$.  
\end{lem}

\begin{proof}
As in the proof of Proposition  \ref{complexregularity}, we may assume $T$ one dimensional.  Consider the 
ideal $\varphi^* (\mathcal I)$ in $\C\{t,s\}$.  The only interesting case is if  $V(\mathcal I)$ is one-dimensional, that is a complex curve germ.  Write its defining function in the form $s^{k _0} h(t,s)$ with  $h$ not vanishing on 
the $t$-axis.  (We do not claim here that  $\varphi^* (\mathcal I)$ is principal.  It may have embedded components at the origin.)  If  $h(0,0)\ne 0$ we are done.  Otherwise, 
 $0$ is an isolated root of $h(t,0)=0$ and then, by  Rouch\'e's Theorem,  $h(t,s)=0$ 
has roots for all $s\ne 0$, that contradicts the assumption that there are no such roots for 
$s\in \R, s>0$.  
\end{proof}

\subsection{Preservation of multiplicity and singular locus} \label{preservationsing} 

In this subsection we suppose that $T$, $Z$, and $Y$ are open subsets of $\K^m$,  $\K^n$, and $\K^{n+m}$, respectively,
 and that $\Ph:T\times Z\to Y$ is an {\aaa } trivialization of the standard projection $\pi : \K^{n+m} \to \K^{m}$. 

 Under these assumptions we first show the preservation of multiplicities of $\Ph$-regular functions.  
 Let us denote $Y_t = Y\cap \pi^{-1}(t)$ and for a function $f: Y\to \K$,  by $f_t$, the restriction of 
$f$ to $Y_t$.   In the following lemma we compare the multiplicities of $f$ at $(t,z)\in Y$ and the multiplicities of the restrictions $f_t$ at $z\in Y_t$.

\begin{prop}\label{regularityequimultplicity}
  If $f: (Y,y_0) \to (\K, 0)$, $y_0=\Ph(t_0,z_0)$,   is  $\Ph$-regular then for $t$ close to $t_0$  the following multiplicities are equal   
\begin{align}\label{genequimultiplicity}
\mult _ {y_0} f  = \mult _ {\Ph (t,z_0)} f= \mult_{y_0} f_{t_0} = \mult_ {\Ph (t,z_0)} f_{t} .
\end{align}
\end{prop}
 
 \begin{proof}
We use the argument of Fukui's proof  of invariance of multiplicity by blow-analytic homeomorphisms, cf. \cite{fukui}.  It is based on the observation that, on a smooth space,  $ \mult_{y_0} f = \min_{y(s)}  \operatorname{ord}_0 f (y(s))$,  where the minimum is taken over all real analytic arcs $y(s) : (I,0) \to (Y,y_0)$.  Thus since $\Ph$ and  
 $\Ph^{-1}$ are arc-analytic 
  $$
\mult_ {\Ph (t,z_0)} f_{t} = \min_{z(s)\to z_0}  \operatorname{ord}_0 f (\Ph (t,z(s))) . 
 $$
For $\Ph$-regular $f$  such orders are preserved by $\Ph$ and this shows the last equality in \eqref{genequimultiplicity}.  The other ones follow from (ii) of Proposition \ref{criterionregularfunction}.  
 \end{proof}

Consider  an ideal  $\mathcal I = (f_1, ...,f_k)$  of $\mathcal O_Y$ and denote by $X=V(\mathcal I)$ its zero set and by  $X_t$ the set $X \cap \pi^{-1}(t)$.  Recall that for $y\in X \subset \K^{n+m}$ the Zariski tangent space $T_{y} X$ is the kernel of the differential $D_{y}(f_1, ...,f_k)$.  

\begin{prop}\label{singularities}
Let $f_i\in \mathcal O_{Y,y_0},  i=1, ... , k$, be $\Ph$-regular and let $X=V(\mathcal I )$.  Then, for every $y$ close to $y_0$, $y= \Ph(t,z)$,  
$T_{y} X_t = \pi^{-1}(t) \cap T_{y} X$ and $\dim_{\K} T_{ \Ph(t,z)} X_t$ is independent of $t$.  In particular,   $Sing X_t = \pi^{-1}(t) \cap Sing X$ and $\Ph$ preserves $Sing X$. 
\end{prop}

\begin{proof}
The equality  $T_{y} X_t = \pi^{-1}(t) \cap T_{y} X$ follows from the fact that the tangent space  
to the leaf through $y$ satisfies $T_yL_z \subset T_{y} X$ and is transverse to the fibers of $\pi$.  

The differential of $f$ at $y$ vanishes if and only if  
 $$
 \min_{y(s)}  \operatorname{ord}_0 f (y(s)) >1
 $$
 where the minimum is taken over all real analytic arc germs $y(s) : (I,0) \to (Y,y)$.  
Similarly, the differentials of $f_1, ... , f_l$ at $y$ are linearly independent  if and only if for every
$i=1, ..., l$ there is a real analytic arc $y(s) : I \to(Y,y)$ such that 
$$
 \operatorname{ord}_0 f_i (y(s)) = 1 \text { and }  \operatorname{ord}_0 f_j (y(s)) >1 \text { for all } j=1, ... ,\hat {\imath}, ... , l. 
 $$
This condition is preserved by $\Ph$.  
\end{proof}



\section{Construction of {\aaa }  trivializations} \label{Puiseux}

In this section we use the Whitney Interpolation and the Puiseux with parameter theorem
 to construct {\aaa } trivializations of equisingular (in the sense of Zariski)
  families of plane curve singularities.  In Part \ref{Part1:Zariski} we will extend this construction  to the Zariski equisingular families of hypersurface singularities in an arbitrary number of variables.  
 
Let 
\begin{align}\label{polynomial}
F(t,x,z) = z^N+ \sum_{i =1}^Nc_i(t,x) z^{N-i}
\end{align}
be a unitary polynomial in $z\in \K$ with $\K$-analytic coefficients $c_i(t,x)$ defined on $U_{\varepsilon,r} = U_{\varepsilon} 
\times U_r $, where $ U_\varepsilon = \{t\in \K^m;  \|t\|< \varepsilon\}$, $U_r= \{x\in \K; | x|<r\}$ .  
Here $t$ is considered as a parameter.   Suppose that  the discriminant  $\Delta_{F_{red}} (t,x)$ of $F_{red}$ 
is of the form 
\begin{align}\label{discriminant}
\Delta_{_{red}F_{red}} (t,x) = x^M u(t,x),  \quad  u\ne 0 \text { on } U_{\varepsilon,r}.
\end{align}

For $F$ reduced, if $M=0$ then by the Implicit Function Theorem the complex roots of $F$, denoted later by 
$a_1 (t,x),  ... , a_N(t,x)$,  are distinct $\K$-analytic functions of $(t,x)$.   In general, by the Puiseux with parameter theorem  they become analytic in $(t,y)$ after a ramification $x=y^d$. 
By Corollary \ref{Puiseuxcor1}, for $x$ fixed, an ordering of the roots at $(0,x)$, $a_1 (0,x),  ..., a_N(0,x)$, gives by continuity an ordering of the roots at $(t,x)$, $a_1 (t,x),  \ldots, a_N(t,x)$.    
Denote by $a(t,x) = (a_1 (t,x),  ... , a_N(t,x))$ the vector of these roots and consider the self-map $\Phi : U_{\varepsilon,r} \times \C 
\to U_{\varepsilon,r} \times \C$ 
\begin{align}\label{phipuiseux}
\Ph (t,x,z ) = (t,x, \ps (z,a(0,x),a(t,x))) , 
\end{align}
where $\psi (z,a,b)$ is the Whitney interpolation map given by  \eqref{ourpsi}.

\begin{thm}\label{psisaa}
For $\varepsilon > 0$ sufficiently small, the  map $\Phi$ defined in \eqref{phipuiseux} is an {\aaa} trivialization of the projection 
$ U_{\varepsilon, r} \times \C \to  U_{\varepsilon}$.  
It preserves the zero set $V(F)$ of $F$ and, moreover, $F$ is $\Ph$-regular along $U_{\varepsilon} \times \{(0,0)\}$.  

If $\K=\R$ then $\Phi$ is conjugation invariant in $z$.  
\end{thm}

Theorem \ref{psisaa} is shown in Subsection \ref{proofpsiaaa}.    
 

\subsection{Puiseux with parameter} \label{puiseux} 
We recall the classical Puiseux  with parameter theorem, see \cite{zariskiI} Thm. 7 and \cite{zariskiII} 
Thm. 4.4, also \cite{pawlucki}.  The Puiseux with parameter theorem is a special case of the Abhyankar-Jung Theorem, 
see \cite{Ab}, \cite{parusinskirond}.  


\begin{thm} \label{PuiseuxTheorem}{\rm (Puiseux with parameter)}\\
Let $F(t,x,z)\in \C \{t,x\}[z]$ be as in \eqref{polynomial}.  Suppose that the discriminant of $F$ reduced   
is of the form  $\Delta_{F_{red}} (t,x) = x^M u(t,x)$ with $u(0,0)\ne 0$.  Then there is a positive integer $d$ 
and $\tilde a_i(t,y)\in \C \{t,y\}$ such that 
$$F(t,y^d,z) = \prod _{i=1}^N(z- \tilde a_i (t,y)) .  
$$  
Let $\theta $ be a $d$th root of unity.  Then for each $i$ there is $j$ such 
that $\tilde  a_i ( t,\theta y) =\tilde a_j (t,y)$.  

If $F(t,x,z)\in \R \{t,x\}[z]$ then the family $\tilde a_i(t,y)$ is conjugation invariant.  
\end{thm}

\begin{cor}\label{Puiseuxcor1}
For $x$ fixed, the  roots of $F$,  $a_1 (t,x),  \ldots, a_N(t,x)$, can be chosen complex analytic  in $t$. 
Moreover,  if $a_i(0,x)=a_j(0,x)$   then $a_i(t,x)\equiv a_j(t,x)$.   Thus the multiplicity of each $a_i(t,x)$ 
as a root of $F$ is independent of $t$. 
\end{cor}

\begin{proof}
It suffices to show it for $F$ reduced.  Then for $x\ne 0$ it follows from the IFT.  Let us show it for $x=0$.  
The family $a_1 (t,0),  \ldots, a_N(t,0)$ coincides with $\tilde a_1 (t,0),  \ldots, \tilde a_N(t,0)$.  
If $\tilde a_i(0,0)= \tilde a_j(0,0)$ then $\tilde a_i(t,y)- \tilde a_j(t,y)$ divides $y^{dM}$ and hence equals 
a power of $y$ times a unit.  
\end{proof}

The following corollary is well-known.  

\begin{cor}\label{Puiseuxcor2}
The Puiseux pairs of $a_i(t,x)$ and the contact exponents between different branches of $V(F)$ are independent of $t$.  
\end{cor}

The next corollary is essential for the proof of Theorem \ref{psisaa}.  It allows us to use the bi-Lipschitz property given by Proposition \ref{lipschitz}.  
Define 
 \begin{align}  \label{bound1}
 \gamma (t,x) &  = \max_{a_i(0,x) \ne a_j(0,x) } \frac {|(a_i(t,x)- a_i (0,x)) -  ( a_j(t,x) -a_j (0,x))|  } 
 {| a_i (0,x)-a_j (0,x) |} \\ \notag 
 & = \max_{a_i(0,x) \ne a_j(0,x) } | \frac {a_i(t,x) -   a_j(t,x)   } 
 { a_i (0,x)-a_j (0,x) }  - 1|. 
\end{align}

\begin{cor}\label{Puiseuxcor3}
There are a positive integer $r$ and positive real constants  $\varepsilon, \delta, C$ such that 
for all $|x|\le \delta$ and $\|t\| \le \varepsilon$ 
 \begin{align*}
\gamma (t,x) \le C\|t\|^r . 
  \end{align*}
\end{cor}

\begin{proof}
We may replace $x$ by $y^d$ and the family $a_i(t,x)$ by complex analytic functions $\tilde a_i(t,y)$.  Suppose that 
$\tilde a_i(t,y)- \tilde a_j(t,y)$ is not  identically equal to zero.  Then, since $\tilde a_i(t,y)- \tilde a_j(t,y)$ divides the discriminant of $F_{red}$, $\tilde a_i(t,y)- \tilde a_j(t,y) = y^{m_{ij}} u_{ij}(t,y)$ with $u_{i,j} (0,0)\ne 0$.  
Therefore $u_{ij}(t,y) -  u_i(0,y)$ belongs to the ideal $(t_1, \ldots , t_m) \C\{t,y\}$.  Consequently there are a positive integer $r_{ij}$ and a constant $C_{ij}$ such that   
 \begin{align*}  
 \frac {|(\tilde a_i(t,y)- \tilde a_j(t,y)) -  (\tilde a_i(0,y)- \tilde a_j(0,y))|  } 
 {|(\tilde a_i(0,y)- \tilde a_j(0,y)) |}  =  \frac {|u_{ij}(t,y) -  u_{ij(}0,y)|  } 
 {|u_{ij}(0,y) |} 
  \le C_{ij} \|t\|^{r_{ij}} 
\end{align*}
in a neighborhood of the origin.  It suffices to take $C=\max C_{ij}$ and $r=\min r_{ij}$.  
\end{proof}


\subsection{Proof of Theorem \ref{psisaa}} \label{proofpsiaaa} 
$\Phi$ is continuous by Proposition \ref{continuity} and Remark \ref{symmetric}.  By Proposition 
\ref{lipschitz} and Corollary \ref{Puiseuxcor3}, if $\varepsilon$ is sufficiently small, then  for 
$t$ and $x$ fixed, $\psi_{a(0,x),a(t,x)} : \C\to \C $ is bi-Lipschitz.  Therefore $\Phi$ is bijective and the continuity of $\Phi^{-1}$ follows from the invariance of domain.   

\begin{lem}\label{regularity} 
For any $r' <r$ there is $C>0$ such that  the restriction 
$\Phi : U_{\varepsilon,r'} \times \C \to U_{\varepsilon,r'} \times \C$ satisfies 
\begin{align}\label{boundF1}
   \quad C^{-1} |F (0,x,z)| \le |F(\Phi(t,x,z))| \le C | F(0,x,z)|.  
\end{align}
\end{lem}

\begin{proof}
 The Lipschitz constants of  $\psi_{a(0,x),a(t,x)} : \C\to \C $ and of its inverse can be chosen  independent of $(t,x)\in U_{\varepsilon,r'}$. Let $L$ be a common upper bound for  these constants.  Then, because 
$\psi_{a(0,x),a(t,x)} (a_i(0,x)) = a_i(t,x)$, 
\begin{align}\label{usingLipschitz}
  L^{-1} |z- a_i(0,x)| \le  | \psi_{a(0,x),a(t,x)} (z) - a_i(t,x)|  \le L|z- a_i(0,x)| 
\end{align} 
Because $F(\Phi(t,x,z)) = \prod_i ( \psi_{a(0,x),a(t,x)} (z) - a_i(t,x))$, we obtain \eqref{boundF1} with $C= L^N$ by taking the product of \eqref{usingLipschitz} over $i$.  
\end{proof}

\changes{
Let us write the formula for $\ps (z,a,b)$ of \eqref{ourpsi},  as   
\begin{align}\label{ourpsi3} 
\ps (z,a,b)  =  z + \frac { Q(z,a) \overline {Q(z,a)} \bigl ( \sum_{j} \sum _k  
  \ Q_{k,j} (z,a,b)  \overline  {Q_k (z, a)}   ( b_j - a_j) \bigr) } 
{N!  Q(z,a) \overline {Q(z,a)}  (  \sum _k  Q_k(z,a)  \overline  {Q_k (z,a ))} } ,
\end{align}
where 
\begin{align*}
& Q_k(z,a) = P_k ((z-a_1)^{-1} , ... , (z-a_N)^{-1}) , \\
& Q_{k,j}(z,a) =
(z-a_j)^{-1} \frac {\partial  P_k}{\partial \xi_j} ((z-a_1)^{-1} , ... , (z-a_N)^{-1})  \\
 & Q(z,a) =  \prod_{s=1}^N (z-a_{s})^{N!} \,  ,
\end{align*}
and the polynomials $P_k$ are defined in Example \ref{exampleinterpol2}.}

  \changes{
The numerator of the fraction in \eqref{ourpsi3}  is a polynomial in the real and imaginary parts of $z,a$, 
and a polynomial  in $b$.  The denominator of this fraction is the a's non-negative real valued polynomial in 
$\re z, \im z, \re a, \im a$.  By Proposition \ref{continuity} this quotient is  continuous on the set $\Xi=\{(z,a,b); \text { if } a_i=a_j \text { then } b_i=b_j\}$.  We show that  $\ps$ is real analytic on the strata of a natural stratification of $\Xi$.  }

The space $\C^N\ni a$ can be stratified by the type of $a$, that is by the number of distinct $a_i$ and by the multiplicities $m_i$ 
they appear in the vector $a$. We encode such a type by the multiplicity vector   
  $\mathbf m=(m_1, . .. , m_d)$, $\sum _{s=1}^dm_i=N$.  
 We denote by  $S_{\mathbf m}\subset \C^N$ the  set of the vectors $a$ with the multiplicity vector $\mathbf m$.  Each stratum, that is each  connected component  of such $S_{\mathbf m}$, is given by 
$S_W=\{a\in S_{\mathbf m} ; a_i=a_j \text{ if } \exists s, \text{ s.t. }  i,j\in W_s\}$, where $W=\{W_s\}$ is a partition  $\{1, . . . ,N\}= \sqcup_s W_s$ with $|W_s| = m_s$.  We denote by $S_W$  the stratum given by partition $W$.

\begin{lem}\label{analytic1}
The restriction of  $\ps (z,a,b)$ of \eqref{ourpsi} to each $\C\times S_W\times S_W$ is real analytic.  
\end{lem}

\begin{proof}
Choose the representatives $i_1, . . . ,i_d$ so that $i_s\in W_s$.  
If we replace in \eqref{ourpsi3}, $Q$ by 
$ Q_W(z,a) = \prod_{s=1}^d   (z-a_{i_s})^{N!}$ then the denominator of the fraction in \eqref{ourpsi3}   does not vanish.  Indeed, first note that for all $k$, $Q_W(z,a)  Q_k(z,a) $ is a polynomial on $\C\times S$.   By property (5) 
of Appendix I, it may vanish only for $z$ equal to one the $a_i$,  say $a_{i_1}$ for instance. Note that 
\begin{align*}
 Q_W(z,a)  Q_{m_1}(z,a)  = \prod_{s=2}^d (z-a_{i_s})^{N!} + (z-a_{i_1}) R(z,a) ,
\end{align*}
where $R$ is a polynomial.  Therefore $ Q_W(a_{i_1},a)  Q_{m_1}(a_{i_1},a) \ne 0$ which suffices to show the claim.  
\end{proof}

For an integer $d$, $1\le d\le N$, we consider 
$$
D_{d}(a) =  \sum_{r_1 <\cdots < r_{d}} \,  \prod_{k< l;\, k,l \in \{r_1, \ldots ,r_{d}\}}  (a_k- a_l)^2.  
$$ 

 \begin{lem}\label{analytic2}
   Let  the germ $a(t,s): (\K^m\times \R ,(0,0)) \to \C^N$ be such that  for every  symmetric polynomial 
   $G$  in $b$,  $G (a(t,s))$ is analytic  in $(t,s)$   (it equals to a power series in $(t,s)\in \K^m\times \R$).  
   We also assume that  for  $s\ne 0$, $a(t,s)$ has exactly $d$ distinct components and that $D_d (a(t,s))$ equals $s^M u(t,s)$ with $u(0,0)\ne 0$.  
Let  $z(s): (\R,0) \to \C$  be  a real analytic germ and  set $a(s) =   a(0,s)$.  Then $\ps (z(s),a(s),a(t,s))$, where 
 $\ps$ is given by \eqref{ourpsi},  is analytic in $(t,s)$.   
 \end{lem} 

 \begin{proof}
 \changes{
 By subtracting $z(s)$ from every component of $a(t,s)$  we may assume that $z(s) \equiv 0$.  
 We consider $a_i (t,s)$ as the roots of a polynomial 
 \begin{align}
 G(z,t,s) = z^N + \sum_{i=1}^N c_i(t,s)z^{N-i}
 \end{align}
 with coefficients analytic in $t,s$.  By Lemma \ref{twodiscr} the discriminant of $G_{red}$ equals a non-zero constant times $D_d (a(t,s))$.  We may  consider $c_i(t,s)$ as complex analytic germs of $(t,s)\in (\C^m\times \C, (0,0))$ and 
 apply to $G_{red}$ the Puiseux with parameter theorem, Theorem \ref{PuiseuxTheorem}.  In particular, for a  fixed $s$,   an ordering of the roots $a_1(s), . . . ,a_N(s)$ of $G(z,0,s)$ gives, by continuity, 
 an ordering of the roots 
 $a_1(t,s), . . . ,a_N(t,s)$ of $G$.  Fix such an ordering and define 
 $$
\varphi (t,s) = \ps (0, a(s), a(t,s)),
$$
where $\ps$ is given by \ref{ourpsi3}.   
Thus defined $\varphi$ is independent of the choice of an ordering (since passing from one ordering to another 
 is given by the action of the same permutation on $a$ and $b$).  
Since $P_k (a) $ is symmetric in $a$, $Q(a)$ and the product $Q(a) Q_k(a)$ are polynomials in the coefficients  $c_i (0,s)$ of $G$.  Hence  $Q(a(s))$ and  $Q(a(s)) Q_k(a(s))$ are complex analytic in $s\in \C$.    

As follows from the next lemma, for a fixed $k$,  $ Q(a(s))(\sum_{j=1}^N  Q_{k,j} (a(s)) (a_j(t,s)-a_j(s)))\in \C\{{t,s}\} $.  
\begin{lem}\label{psisaa1}
Let $P(a,b)\in \C[a,b]$  be a polynomial invariant under the action of the permutation group: 
$P  (\sigma (a),\sigma (b)) = P (a,b) $ for all $\sigma \in S_N$.    Then $P(a(s),a(t,s))\in \C\{t,s\}$.    
\end{lem}  
\begin{proof} 
We may assume that  $P(a(s),a(t,s))$ is well-defined for $(t,s) \in B\times D$, 
 where $B$ is a neighborhood of the origin in $\C^m$   and $D$ is a small disc centered at the origin 
 in  $\C$.  By the assumption $P(a(s),a(t,s))$ is bounded and complex analytic on $B\times (D\setminus \{0\})$.  
 Therefore it is complex analytic on $B\times D$.  
\end{proof}

In particular, by Lemma \ref{psisaa1}, the numerator of the fraction in  \eqref{ourpsi3}, evaluated on $a(t,s),a(s)$ is analytic in 
$(t,s) \in \K^m\times \R$.  
As we have shown before its denominator is analytic in (one variable) $s\in \R$.  Therefore,  $\varphi(t,s)$ is of the form 
$s^{-k}$ times a power series in $(t,s)$.  Since, moreover,  $\varphi(t,s)$ is bounded  in a neighborhood of the origin it 
 has to be analytic.  }
\end{proof}

 It follows from Lemma \ref{analytic1} that 
$\ps (z,a(0,x),a(t,x))$ of \eqref{phipuiseux} is real analytic on $x\ne 0$ and on $x=0$ and from Lemma  
\ref{analytic2} that it is arc-wise analytic.   
The next lemma shows that the inverse of $\Ph$ is arc-analytic  and completes the proof of 
Theorem \ref{psisaa}.  

\begin{lem}\label{psiinvsaa}
If $(t(s), x(s), z(s))$ is a real analytic arc,  then  there is a real analytic $\tilde z(s)$ such that 
$(t(s), x(s), z(s)) = \Phi (t(s), x(s), \tilde z(s))$.    
\end{lem}  

\begin{proof} 
Since $\Ph^{-1}$ is subanalytic such $\tilde z(s)$ exists continuous and subanalytic.  
Thus there is  a positive integer $q$ such that for $s\ge 0$,  $\tilde z(s)$ is a convergent power series in $s^{1/q}$.  We show that  all exponents of $\tilde z(s)$, $s\ge 0$,   are integers.    Suppose that this is not the case.  Then 
$$
\tilde z(s) =  \sum_{i=1}^n v_i s^i + v_{p/q} s^{p/q} + \sum_{k> p}  v_{k/q} s^{k/q} ,
$$
with $p/q>n$ and $p/q\not \in \N$.  Denote  $\tilde z_{an} (s) = \sum_{i=1}^n v_i s^i $.  
Then $\psi (\tilde z_{an} (s), a(0,x(s)), 
a(t(s),x(s)))$ is real analytic and by the bi-Lipschitz property, Proposition \ref{lipschitz}, 
$$
|\psi (\tilde z_{an} (s), a(0,x(s), a(t(s),x(s))) - \psi (\tilde z (s), a(0,x(s)), a(t(s),x(s)))| \sim 
|  \tilde z_{an} (s) - \tilde z(s)| \sim s^{p/q},
$$
that is impossible since $\psi (z_{an} (s), a(0,x(s)), a(t(s),x(s)))$ and $\psi (z (s), a(0,x(s)), a(t(s),x(s)))$ are real analytic in $s$.   

This shows that $(t(s), x(s), z(s))$, $ \Phi (t(s), x(s), \tilde z(s))$ are two real analytic arcs that coincide for 
$s\ge 0$ and therefore also for $s\le 0$.   
\end{proof}

\subsection{Preservation of multiplicities  of roots} \label{preservationsec} 

Corollary \ref{Puiseuxcor1}  admits a multidimensional generalization, see Zariski \cite{zariski1975}.  
 In the sequel  we will need   the following result that is a consequence of \cite{zariski1975}  and Proposition 
\ref{regularityequimultplicity}.  We include its proof for the reader's convenience.

 \begin{lem}[Preservation of multiplicities of roots] \label{preservationlem}
Let $\Ph:T\times Z\to Y$ be an {\aaa } trivialization, $y_0=\Ph(t_0,z_0)$, and let $A_i$, $i=1,...,N$,  
be $\K$-analytic functions defined in a neighborhood of $y_0$.   Let 
  \begin{align*}
f(y,w) =  w^{N}+ \sum_i A_i(y) w^{N-i}
\end{align*}
and suppose that the discriminant $\Delta (f_{red})$  is $\Ph$-regular.  Then, for $t$ in a neighborhood of $t_0$,
 the roots of $f$ at $\Ph (t,z_0)$,  
 $$a_1 ( \Ph(t,z_0)),  \ldots, a_N (\Ph(t,z_0)),$$  can be chosen complex analytic in $t$.  
 (Moreover, any continuous choice is complex analytic.) For such a choice, if $a_i ( \Ph(t_0,z_0)) = 
 a_j ( \Ph(t_0,z_0))$   then $a_i ( \Ph(t,z_0)) =
 a_j ( \Ph(t,z_0))$ for all $t$.   In particular,  the multiplicity of each $a_i (\Ph(t,z_0))$ 
 as a root of $f$ is independent of $t$.   
\end{lem}
 
 \begin{proof}
Choose a real analytic arc germ $z(s) :  I \to Z$, $z(0)=z_0$, so that  $\Delta (f_{red})$ is 
not identically zero on $\Ph (t,z(s))$.  By Corollary \ref{Puiseuxcor1} it suffices to show that 
$F(t,s,w)=  f_{red} (\Ph (t,z(s)), w  )$ satisfies the assumptions of the Puiseux with parameter theorem.  
To show it we first note that the  discriminant of $F$ equals to $\Delta (f_{red}) (\Ph (t,z(s)))$.  Secondly, we 
observe that, by regularity of $\Delta (f_{red})$ on $z(s)$  in the form \eqref{regularfunctioncondition}, 
$\Delta (f_{red}) (\Ph (t,z(s)))$ equals $s^k$ times an analytic unit. 
 \end{proof}


{\Large \part {Zariski Equisingularity.}\label{Part1:Zariski}
}
\medskip
\section{Zariski Equisingularity implies {\aaa} triviality.}

In this section we generalize Theorem \ref{psisaa} to an arbitrary number of variables
 hypersurface case.  

\begin{defi}\label{system}
 By a \emph{local system  of pseudopolynomials in $x=(x_1,... ,x_n)\in \K^n$ at $(0,0)\in \K^m\times \K^n$, with parameter $t\in U \subset \K^m$}, we mean a family  of analytic  functions 
  \begin{align}\label{polynomials}
 F_{i} (t, x_1, \ldots, x_i  )= x_i^{d_i}+ \sum_{j=1}^{d_i} A_{i,j} (t,x_1, \ldots, x_{i-1}) x_i^{d_i-j}, \quad i=0, \ldots,n, 
\end{align} 
defined on $U\times U_i$, where $U$ is a neighborhood of the origin in $\K^m$,  $U_i $ 
is a neighborhood of the origin in $\K^i$, with the coefficients $A_{i,j}$ vanishing identically on $T=U\times \{0\}$.  
Thus $F_0$ is an analytic function depending only on $t$.  We also assume that, for each $i=1,\ldots,n$,   
the discriminant of $F_{i,red}$ divides  $F_{i-1}$.  
For  $d_i = 0$,  by \eqref{polynomials} we mean that $F_i\equiv 1$, and in this case by convention we define all $F_j, j<i$, as identically equal to $1$.   
\end{defi}

We call this system   \emph{Zariski equisingular} if $F_{0} (0) \ne 0$.   
As Varchenko showed in \cite{varchenko1972}, answering a question posed by Zariski in \cite{zariski},  
 for a Zariski equisingular system, the family of analytic set germs 
$X_t=\{ F_n(t, x) =0\}\subset (\K^n,0)$ is topologically equisingular for $t$ close to the origin.   
In this section we show that this equisingularity can be obtained by an {\aaa } trivialization. 

 \begin{rem}
The above definition is slightly more general than that of  \cite{zariski} or \cite{varchenko1972} where it is  
 assumed that $F_{i-1}$ is the Weierstrass polynomial associated to  the discriminant of $F_{i,red}$.   
Our less restrictive assumption is sufficient for the proof of  Theorem \ref{theorem}.  In fact, in 
 the inductive step  we only need that  the discriminant of $F_{i,red}$ is  $\Ph_{i-1}$-regular for the trivialization $\Ph_{i-1}$.  
 By Proposition \ref{criterionregularfunction}  this is the case if this discriminant divides $F_{i-1}$ and $F_{i-1}$ is 
$ \Ph_{i-1}$-regular.  
 \end{rem}
 

\begin{thm}\label{theorem}
If $F_i(t,x) $, $i=0, \ldots, n$, is a Zariski equisingular local system of pseudopolynomials,      
then there exist $\varepsilon >0$ and a  
homeomorphism 
\begin{align}
\Ph  : B_{\varepsilon} \times \Omega_0 \to \Omega, 
\end{align}
where $B_{\varepsilon} = \{t\in \K^m; \|t\|<\varepsilon\}$, $\Omega_0$ and $\Omega$ are neighborhoods of 
the origin in $\K^n$ and $\K^{m+n}$ resp.,  such that
 \begin{enumerate}
\item  [\rm (Z1)]
$\Ph (t, 0) =  (t,0) $,  $\Ph (0, x_1, \ldots, x_n) =  (0, x_1, \ldots, x_n) $;
\item [\rm (Z2)]
$\Ph$ has a triangular form 
\begin{align*}
\Ph (t, x_1, \ldots , x_n) = (t, \Ps_1(t, x_1), \ldots , \Ps_{n-1} (t,x_1, \ldots , x_{n-1}), \Ps_{n} (t, x_1, \ldots , x_{n}) );
\end{align*}
\item [\rm (Z3)]
For $(t, x_1, \ldots , x_{i-1})$ fixed, 
$\Ps_i (t, x_1, \ldots , x_{i-1}, \cdot ): \K \to \K$ is bi-Lipschitz  
and the Lipschitz constants of 
$\Ps_i$ and $\Ps_i ^{-1}$ can be chosen independent of $(t, x_1, \ldots , x_{i-1})$;  
\item [\rm (Z4)]
$\Ph$ is an {\aaa} trivialization of the standard projection $\Omega \to B_\varepsilon$;  
\item [\rm (Z5)]
$F_n$ is regular along $B_{\varepsilon} \times \{0\}$.  
\end{enumerate}
\end{thm}

Recall after  Proposition \ref{regularproducts} that (Z5) implies that for any analytic $G$ dividing a power of $F_n$, there is $C>0$ such that 
\begin{align}\label{boundforG}
&  C^{-1} |G (0,x)| \le |G (\Ph (t,x))| \le  C |G (0,x)|.  
\end{align}
  In particular $\Phi$ preserves the zero level of $G$.  

\begin{rem}\label{recursiveargument} Strategy of proof.  \\
The functions $\Ps_i$ will be constructed inductively so that  every  
\begin{align}\label{Phiinduction} 
\Ph_i (t,x_1, \ldots , x_i) = (t, \Ps_1(t, x_1), \ldots , \Ps_{i} (t,x_1, \ldots , x_{i}) )
\end{align}
 satisfies the above  properties (Z1)-(Z4) and (Z5) for $F_i$.   Given $\Ph_{n-1}  : B_{\varepsilon'} \times \Omega'_0 \to \Omega'$.  
 We first lift it (by continuity) to all complex  roots of $F_n$, then we extend it  to $ B_{\varepsilon'} \times \Omega'_0 \times \C$ 
 by the Whitney Interpolation Formula. The fact that the  trivialization $\Ph(t,x)$ obtained in this way is 
{\aaa } is proven by a reduction to the Puiseux with parameter case as follows. Let  
$x(s) = (x'(s), x_n(s)) $ be a real analytic arc.  By the inductive assumption  $\Ph_{n-1} (t,x'(s))$ 
is analytic in $t,s$.    We show that  $f(t,s,z) = (F_{n}(\Ph_{n-1} (t, x' (s)), z))_{red}$  satisfies the assumptions of 
the Puiseux with parameter theorem, and then we conclude by Theorem \ref{psisaa}.  
We first consider $x'(s)$ sufficiently generic, so that the discriminant of $F_{n,red} (\Ph_{n-1} (t,x'(s)),z)$ 
does not vanish identically, and use this case to show that the number and the multiplicities of the roots of 
$F_n$ are constant over each leaf of $\Ph_{n-1}$. This will imply the case of arbitrary arcs $x(s)$.  

The fact that $\Ph$ satisfies the property  (5) of Definition \ref{maindefinition2} will be shown later in Section \ref{associatedstratification} where the appropriate stratification is introduced.  
In the argument below we do not use this property in the inductive step.
\end{rem}

\begin{proof}
The proof is by induction on $n$.
Thus suppose that  $\Ps_1, \ldots, \Ps_{n-1}$ are already constructed    and that  for $i<n$  the homeomorphisms \eqref{Phiinduction}   satisfy the properties (Z1)-(Z5).    
 To simplify the notation we write $(x_1,\ldots,x_{n}) = (x',x_n) $.  
 By the inductive assumption $F_{n-1}$, and hence by Proposition \ref{regularproducts} the 
 discriminant of $F_{n,red}$, is regular for  $\Ph_{n-1}  : B_{\varepsilon'} \times \Omega'_0 \to \Omega'$.    
Therefore, by the preservation of multiplicities of roots principle, Lemma \ref{preservationlem}, for any $x'\in \Omega'_0 $, the complex roots of   $F_n $  
 $$a_1(\Ph_{n-1}(t,x')) , \ldots , a_{d_n}(\Ph_{n-1} (t,x'))$$ 
 can be chosen  $\K$-analytic in $t$.  Moreover,  $a_i(0,x')= a_j(0,x')$ if and only if $a_i(\Ph_{n-1}(t,x'))= a_j(\Ph_{n-1}(t,x'))$ for all $t\in B_{\varepsilon'}$.   Denote by  $a(\Ph_{n-1}(t,x')) =(a_1(\Ph_{n-1}(t,x')) , \ldots , a_{d_n}(\Ph_{n-1} (t,x'))) $ 
 the vector of such roots and  set  
\begin{align}\label{definitionPsi}
\Ps_{n} (t,x) : & =  \ps ( x_n, a(0,x'), a(\Ph_{n-1}(t,x'))) \\
& = x_n + \frac {\sum_{j=1}^N \mu_j (x_n,a(0,x')) (a_j(\Ph_{n-1}(t,x'))- a_j(0,x'))} {\sum_{j=1}^N \mu_j (x_n,a(0,x'))}   ,  \notag
\end{align}
where $\ps$ is given by \eqref{ourpsi}, and then define $\Phi$ by (Z2).  

Thus constructed $\Phi$ satisfies (Z1) and (Z2) by its definition.  We show that $\Ph$ is 
a homeomorphism that  satisfies
 (Z3)-(Z5).  
This we check on every real analytic arc applying  the Puiseux with parameter theorem.
 
 \begin{lem}
 Let $ K\Subset \Omega'_0$.  Then 
   \begin{align}\label{Lipschitzonroots}
\sup_{x'\in  K}  \max_{a_i (0,x')\ne a_j (0,x')} \frac { |a_i(\Ph_{n-1}(t,x') ) -a_j(\Ph_{n-1}(t,x') ) |}{|a_i (0,x')-a_j (0,x')|} \to 1 \quad \text {as} \quad t \to 0
\end{align}
 \end{lem}
 
 \begin{proof}
Denote 
 \begin{align*}
 \gamma (t, x')  = \max_{a_i (0,x')\ne a_j (0,x')} \frac { |(a_i(\Ph_{n-1}(t,x') )  -a_j(\Ph_{n-1}(t,x') ) )- (a_i(0,x') - a_j (0,x'))|}
 {|a_i (0,x')-a_j (0,x')|}.  
\end{align*}
We show that $\gamma$ is  bounded  on $B_{\varepsilon'} \times  K$, after replacing $\varepsilon'$ by a smaller 
positive number if necessary, and converges to $0$ as $t$ 
goes to $0$.    Let $x'(s)$ be a real analytic arc such that $(0,x'(s))$ is not entirely included in 
 the zero set of $F_{n-1}$.   By Corollary \ref{Puiseuxcor3}, $\gamma$ is bounded on $(t, x'(s))$ 
 and converges to $0$ as $t$ goes to $0$.  Thus, by the curve selection lemma, the claim holds on 
  $\{(t,x'); F_{n-1} (t,x') \ne 0\}$.  We extend it on the zero set of $F_n$ by the lower 
  semi-continuity of $\gamma$, Remark \ref{gammalimit}.   
\end{proof}

Thus, taking $\varepsilon'$ smaller if necessary, we see by Proposition \ref{lipschitz} that  $\Ps_n$ of \eqref{definitionPsi} is well-defined, continuous by Proposition \ref{continuity}, 
and satisfies  (Z3).    

Choose a neighborhood $\tilde \Omega'_0$ of the origin  in $\K^{n-1}$, $\varepsilon \le \varepsilon'$  and 
$r>0$ so that $ \tilde \Omega'_0 \Subset \Omega'_0$ and  $F_n$ does not vanish on $B_{\varepsilon} \times  \tilde \Omega'_0 \times \partial D$, where $\partial D= \{x_n \in \K; \|x_n\|=r\}$.  
Then we set $\Omega_0=  \tilde \Omega'_0 \times D$, where  $D=\{x_n \in \K; \|x_n\|<r\}$,  and $\Omega = \Ph (B_\varepsilon \times \Omega_0)$.

Now we show (Z4) 
(except the property  (5) of Definition \ref{maindefinition2} that will be shown in Section \ref{associatedstratification}). 
 Let $x(s): I \to  \Omega_0$ be a real analytic arc.  We show that 
$ \Ph (t, x(s) )$ is analytic in $t$ and $s$.    If $(0,x'(s))$ is not entirely 
included in the zero set of $F_{n-1}$ then it follows from  Theorem
 \ref{psisaa} (we argue as in the proof of  Lemma \ref{preservationlem}).   Thus, suppose  $F_{n-1}(0, x'(s))\equiv 0$.   Consider 
\begin{align}\label{redonacurve}
f(t,s,z) = (F_n(\Ph_{n-1} (t, x' (s)), z))_{red} .
\end{align}
By \eqref{Lipschitzonroots} the size of the discriminant $\Delta_f(t,s)$ of $f$ is 
independent of $t$, that is there are  constants $C, c>0$ such that 
$$c|\Delta_f(0,s)\le |\Delta_f (t,s)|\le  C|\Delta_f(0,s).$$
Write $\Delta_f$ in the form $s^M h(t,s)$, where $h$ does not vanish identically on $s=0$.  By the 
above inequality  we conclude that $h(0,0)\ne 0.$  Hence $f(t,s,z)$ satisfies the assumption of Theorem \ref{psisaa} 
that implies that $ \Ph (t, x(s) )$ is analytic in $t$ and $s$.

To show that  the inverse of $\Ph$ is arc-analytic we use the inductive assumption, i.e. the assumption that the 
inverse of $\Phi_{n-1}$, is arc-analytic.   Then, for a real analytic arc $x'(s)$ fixed,   
 over its flow $(t,s) \to \Phi_{n-1} (t ,x'(s))$, we use Lemma \ref{psiinvsaa}.  

The proof of (Z5) is similar to that of (Z4).  First, by Proposition \ref{criterionregularfunction}, 
  it suffices to show it over the flow of any real analytic arc  
$x'(s)$, that is for  $(t,s) \to \Phi_{n-1} (t ,x'(s))$.    If $(0,x'(s))$ is not entirely 
included in the zero set of $F_{n-1}$, then it follows directly from the proof of  
Lemma \ref{preservationlem}  and Theorem
 \ref{psisaa}.   If $F_{n-1}(0, x'(s))\equiv 0$ then we consider \eqref{redonacurve} and conclude 
 again by Theorem \ref{psisaa}.
 \end{proof}

\subsection{Geometric properties} In this subsection we summarize some geometric properties of 
the {\aaa } trivialization $\Ph$ constructed in the proof of Theorem \ref{theorem}.  
Firstly, $\Ph$ preserves the multiplicities and the singular loci of the $\Ph$-regular functions.  

The preservation of multiplicity follows by induction from Zariski \cite{zariski1975}, or, independently from Proposition \ref{regularityequimultplicity}. 

\begin{prop}[Zariski equisingularity implies equimultiplicity]\label{equimultiplicity}
Let $F_i$, $i=0, \ldots, n$, be a Zariski equisingular local system of pseudopolynomials 
at the origin in $\K^m\times \K^n$.  Then for any $\K$-analytic  function $G$ dividing $F_n$, the multiplicities 
\begin{align}
 \mult _ { (t,0)} G= \mult_{0} G_{t} ,
\end{align} 
where $G_t (x) = G(t,x)$, are independent of $t$.  \qed
 \end{prop}
 
Note that, by construction  $\Ph(t,x) = (t,\Ps (t,x))$ is real analytic in the complement   of   $B_\varepsilon \times Z$,  where $Z$ is a nowhere dense $\K$-analytic  subset of $ \Omega_0 $.  Let us, for  $t$ fixed, denote 
$x\to \Ps(t,x)$ by $\Ps_t$.  
It follows from (Z2) and (Z3) that the jacobian determinant of $\Ps_t$, that is well-defined in the complement 
of   $B_\varepsilon \times Z$,    is bounded from zero 
and infinity in a neighborhood of the origin, that is there exists $C,c>0$ such that 
 \begin{align}\label{jacdet} 
c\le | jac\det (\Ps_t) (t,x) |\le C . 
\end{align} 

Consider an analytic set $X= \{f_1(t,x)= ... =f_k(t,z)=0\}\subset \Omega$ defined by 
$\K$-analytic $\Ph$-regular functions $f_1(t,x), ... , f_k(t,z)$.  Denote $X_t= X \cap \pi^{-1}(t)$.  
Then, as follows from Proposition \ref{singularities}, $Sing X_t = \pi^{-1}(t) \cap Sing X$ and 
$\Ph$ preserves $Sing X$ and $Reg X$.  


\subsection{Generalizations} \label{additional}
 
 The following generalization can be used to show the topological equisingularity of analytic function 
 germs, see \cite{BPR} and Subsection \ref{trivialityfunctions} below.  
 
 \begin{prop}\label{addendum}
  Theorem \ref{theorem} holds if in the definition of a local system of pseudopolynomials  the assumption 
\begin{enumerate}
\item[(i)]
The discriminant of $F_{i,red}$ divides  $F_{i-1}$.  
 \end{enumerate}
 is replaced by 
 \begin{enumerate}
 \item[(ii)]
There are $q_i\in \N$ such that $F_{i}=x_1^{q_i} \tilde F_i $, where $\tilde F_i (x_1,...,x_i)$ is a monic  Weierstrass polynomial in $x_i$, and for $i=1, ... , n$,  the discriminant of $\tilde F_{i,red}$ divides $F_{i-1}$. 
 \end{enumerate}
Moreover, in the conclusion we may require that $\Ps_1 (t,x_1)\equiv x_1$.  
\end{prop}

\begin{proof}
We can always require $\Ps_1 (t,x_1)\equiv x_1$ in the first step of construction.  Then, in the 
inductive step, we assume that  $x_1$ and $\tilde F_{n-1}$ are $\Ph_{n-1}$-regular.  Hence, by
Proposition \ref{regularproducts}, so is the discriminant of $\tilde F_{i,red}$.  This allows us to proceed with the 
construction of $\Ph$.  Since $x_1$ is constant on the fibers of $\Ph$, it is $\Ph$-regular and $\tilde F_{n}$ 
is $\Ph$-regular by the proof of Theorem \ref{theorem}.
\end{proof}



\section{Zariski Equisingularity with transverse projections. } \label{transverseprojections}

\begin{defi}\label{transversedefinition}
We say that a local system of pseudopolynomials $F_i(t,x)$, $i=1, \ldots, n$, is  \emph{transverse}
at the origin in $\K^m\times \K^n$, if for every  $i=2, \ldots, n$, the multiplicity $\mult_0 F_i (0,x)$ of $F_i(0,x)$ at $0\in \K^i$ is equal to  $d_i$.  
\end{defi}

We always have the upper semi-continuity condition.  If we denote $F_t(x) = F(t,x)$, then  $\mult_{0} F_t \le \mult_0 F_0$
 for $t$ close to $0$.  Since $\mult_{0} F_t \le d_i$,  the  transversality is a closed condition (in the Euclidean or analytic Zariski topology) in parameter $t$. 

 If the system $\{F_i\}$  is Zariski equisingular then, by Proposition \ref{equimultiplicity}, 
the transversality is also an open condition.  Thus in this case the system is transverse at any $(t,0)\in U$, 
keeping the notation from Definition \ref{system}.  Therefore, writing $F$ instead of $F_n$, 
 we have $d_n= \mult_{0} F_t = \mult_0 F_0$ and also $d_n= \mult_{(0,0)} F = \mult_{(t,0)} F$.  

Denote $X=F^{-1}(0)$, $X_t= X\cap \{t\} \times \K^n$.  Geometrically the assumption $\mult_0 F (0,x)= d_n$ means that the kernel of the standard projection 
$\K^n \to \K^{n-1}$ is transverse to the tangent cone of $X_0$ at the origin, i.e. the vertical line $\{0\}\times 
\K \subset \K^{n-1} \times \K$ is not entirely included in  this tangent cone.  
If this is the case then, in the Zariski equisingular case, by Proposition \ref{equimultiplicity}, the kernel of the standard projection 
$\pi :\K ^m \times \K^n \to \K ^m \times  \K^{n-1}$ is  transverse to the tangent cone of $X$ at the origin.

\changes{
\begin{defi}\label{partiallytransversedefinition}
We say that a local system of pseudopolynomials $F_i(t,x)$, $i=1, \ldots, n$, is  \emph{partially transverse}
if each $F_i$ with $d_i>0$  has a factor $G_i$ of degree $d'_i>0$ in $x_i$ such that  
 $\mult_0 G_i (0,x)=d' _i$.  
\end{defi}}

It is clear from the definitions  that a transverse system is partially transverse.

\begin{thm}\label{theoremtransverse}
Let $F_i(t,x)$, $i=0, \ldots, n$, be a Zariski equisingular local system of pseudopolynomials partially 
transverse  at the origin in $\K^m\times \K^n$.     
 Let $\Ph (t,x) = (t, \Ps (t,x))  : B_{\varepsilon} \times \Omega_0 \to \Omega$ 
be the homeomorphism  constructed in the proof of Theorem \ref{theorem}.  Then    
 \begin{enumerate} 
\item [\rm (Z6)] $\Ph$ is an {\aaa} trivialization regular along  $B_\varepsilon \times \{0\}$.    
\end{enumerate}
\end{thm}

\begin{proof}
We have to show, see Subsection \ref{regular},  that, after shrinking the neighborhood $\Omega$ if necessary,  there is a constant $C>1$ such that   for all $(t,x) \in B_{\varepsilon} \times \Omega_0$,   
\begin{align}\label{distancebound}
  C^{-1} \|x\| \le \|\Ps (t,x))\| \le  C \|x\|  .
\end{align}
This will be shown by induction on $n$.    Let us write for short  $x = (x', x_n)$ and 
$\Ph (t,x) = (t, \Ps (t,x)) = (t, \Ps'(t,x') , \Ps_n(t,x))$.  
 By the inductive assumption  
\begin{align}\label{developmentinduction}
  C_1^{-1} \|x'\| \le \|\Ps' (t,x'))\| \le  C_1 \|x'\|  .
\end{align}

Let $a_1(t,x') , \ldots , a_{d'_n} (t,x'))$ 
denote  the complex  roots of $G_n = x_n^{d'_n} + \sum A'_{j} (t,x') x_n^{d'_n-j} $, where $G_n$ is given by Definition \ref{partiallytransversedefinition}. 
By the assumption on $G_n$,  $|A'_{j} (t,x')|\le C_2\|x'\|^j$, 
for all $j=0, ..., d'_n-1$,   and hence these roots satisfy
$|a_{i} (t,x')|\le C_3\|x'\|$.  The latter bound, by the inductive assumption, is equivalent to 
\begin{align}\label{roots}
|a_{i} (t,\Ps'(t,x'))|\le C_4\|x'\|.
 \end{align} 

By  formula \eqref{definitionPsi},
$\Ps_{n} (t,x) :  =  \ps ( x_n, a(0,x'), a(\Ph_{n-1}(t,x'))) $ and $ \ps ( a_i(0,x'), a(0,x'), a(\Ph_{n-1}(t,x')))
=a_i(t,\Ps' (t,x'))$ and therefore by
 the Lipschitz property of Whitney Interpolation, Proposition 
\ref{lipschitz},  we get
 \begin{align}\label{distancetoroots}
C^{-1}_5  | x_n  - a_i (0,x')| \le | \Psi_n (t,x) - a_i(t,\Ps' (t,x'))| \le C_5  | x_n  - a_i (0,x')| .
\end{align}

By \eqref{roots} and \eqref{distancetoroots}
 \begin{align*}
| \Psi_n (t,x)| \le C_6 ( | x_n  - a_i (0,x')|  + | a_i(t,\Ps' (t,x'))| ) \le  C_7 \|(x',x_n)\| 
\end{align*}
that shows the second inequality in \eqref{distancebound}.  The proof of the first one is similar.  
 
  This ends the proof of  Theorem \ref{theoremtransverse}.  
\end{proof}

\begin{example}\label{derivationcomplete}
\changes{ Let $\mathcal G_n =\{G_{n,i} (t,x)\}$ be a finite family of monic pseudopolynomials in $x_n$.  We say that 
$\mathcal G$ \emph{is stable by derivation}  if  for every $G\in \mathcal G_n$,  either 
$\partial G /\partial x_n\equiv 0$ or $\partial G /\partial x_n\in \mathcal G_n$ (after multiplication by a non-zero constant).  We say that a pseudopolynomial $F_n$ is 
\emph{derivation complete} if it is the product of a stable by derivation family.   We call a system of pseudopolynomials $\{F_i\}$   \emph{derivation complete}  if so is every  $\{F_i\}$.}

\changes{
Suppose now that  the system $\{F_i\}$  is derivation complete and let $F_i$ be the product of a stable by derivation family  $\mathcal G_i =\{G_{i,j} \}$.  If $F_i(0,0)=0$ then there is $G\in \mathcal G_{i}$, 
such that $G(0,0)=0$, $ \partial G/\partial x_i(0,0)\ne 0$.  The Weierstrass polynomial associated to $G$ is of degree $1$ in $x_i$ and divides $F_i$.
  Hence the family $\{F_i\}$ is partially transverse.  }
\end{example}


\section{Canonical stratification associated to a system of pseudopolynomials.}
\label{associatedstratification}

In this section we extend the results of the last two sections to a more global situation.

\begin{defi}\label{systempseudopolynomials}
By a \emph{system of pseudopolynomials in $x=(x_1,... ,x_n)\in \K^n$} we mean a family 
 \begin{align}\label{polynomials2}
F_{i} (x_1, \ldots, x_i  )= x_i^{d_i}+ \sum_{j=1}^{d_i} A_{i,j} (x_1, \ldots, x_{i-1}) x_i^{d_i-j}, \quad i=1, \ldots,n, 
\end{align}
with $\K$-analytic coefficients $A_{i,j}$, satisfying 
\begin{enumerate} 
\item 
there are   $\varepsilon _j\in (0,\infty]$, $j=1, \ldots, n$,  such that every $F_i$ are defined on 
$U_i = \prod_{j=1}^i D_j$, 
where $D_j = \{  |x_j|< \varepsilon _j  \} $.
\item 
if $\varepsilon_i<\infty$ then $F_{i} $ does not vanish on  $U _{i-1} \times \partial D_i$, where   
$\partial D_i = \{  |x_i|= \varepsilon _i  \} $.  
\item
for every $i$, the discriminant of $F_{i,red}$ divides   $F_{i-1}$.   
\end{enumerate}
It may happen that $d_i=0$.  Then $F_i\equiv 1$ and we set by convention $F_j\equiv 1$ for $j<i$.  

We say that $\{F_i\}$ is \emph{a system of polynomials} if every $F_i$ is a polynomial.  
\end{defi}

 For $i<k$ we denote by $\pi_{k,i} : U_k \to U_i$ the standard projection.   For each $i$ we define a  filtration
\begin{align}\label{assstrat}
U_i = X^i_i \supset X^i_{i-1} \supset \cdots \supset X^i_0, 
\end{align}
where
\begin{enumerate} 
\item 
$X^1_0 = V(F_1)$.  It may be empty.  
\item 
$X^i_j =  (\pi_{i,i-1}^{-1} (X^{i-1}_j ) \cap V(F_i) )\cup \pi_{i,i-1}^{-1} (X^{i-1}_{j-1})$ for $1\le j<i$.    
\end{enumerate}

As we show below  every connected component $S$ of $X^i_j \setminus X^i_{j-1}$ is a locally closed  
$j$-dimensional $\K$-analytic submanifold of $U_i$ and hence \eqref{assstrat} defines an analytic stratification $\mathcal S_i$ of $U_i$, {see Section \ref{stratifications} for the definition.  We call $\mathcal S = \mathcal S_n$ \emph{the canonical stratification associated to a system of pseudopolynomials}.

\begin{prop}\label{canonicalstratification}
For all  $j\le i\le n$ every connected component $S$ of $X^i_j \setminus X^i_{j-1}$ is a locally closed  $j$-dimensional $\K$-analytic submanifold of $U_i$ of one of the following two types: 
\begin{enumerate} 
\item [\rm (I)]
$S\subset V(F_i)$ and there is a connected component $S'$ of $X^{i-1}_j \setminus X^{i-1}_{j-1}$ such that 
$\pi_{i,i-1}$ induces a finite  $\K$-analytic covering $S\to S'$.  
\item [\rm (II)]
There is  a connected component $S''$ of $X^{i-1}_{j-1} \setminus X^{i-1}_{j-2}$ such that   $S$ is a connected component of $\pi_{i,i-1}^{-1}  (S'') \setminus V(F_i)$.  
\end{enumerate}
Moreover, for every $p\in S$ there are a local system of coordinates  at $p$ in which $(S, p)=(\K^j,0)$,  neighborhoods $B$, $ \Omega_0$ and $ \Omega$ of $p$ in $\K^j$, $\K^{i-j}$, and $\K^i$ resp., 
and an {\aaa } trivialization 
$$\Ph :B\times  \Omega_0\to   \Omega$$
preserving   the strata of stratification $\mathcal S_{i}$ and such that $F_{i}$ is $\Ph$-regular.  
If the system $\{F_i\}$  is derivation complete in the sense of Example \ref{derivationcomplete} 
 then the trivialization $\Ph$ can be chosen 
 regular  along 
$B\times \{p\}$.
\end{prop}

\begin{proof}
Induction on $n$.   Let $S'$ be a  stratum of $\mathcal S_{n-1}$ of dimension $j$ and let $p'\in S'$.  
  By the inductive assumption there are a local system of coordinates $y_1, ..., y_{n-1}$
 at $p'$ in which $(S', p')=(\K^j,0)$,  neighborhoods $B'$, 
$ \Omega'_0$ and $ \Omega'$ of $p'$ in $\K^j$, $\K^{n-1-j}$, and $\K^{n-1}$ resp.,  
and an {\aaa } trivialization 
$$\Ph' :B'\times  \Omega'_0\to   \Omega'$$
preserving   $\mathcal S_{n-1}$,  such that $F_{n-1}$ is  $\Ph'$-regular.  
 Since   the discriminant of $F_{n,red}$ divides $F_{n-1}$ it is also $\Ph'$-regular.    
Therefore, by Lemma \ref{preservationlem}, the restriction of projection $\pi_{n,n-1}$
$$\pi_{n,n-1} ^{-1} (B') \cap V(F_n) \to B'$$
 is a finite 
analytic covering.  This shows that the connected components of  
$ \pi_{n,n-1}^{-1} (S') \cap V(F_n)$ and of $ \pi_{n,n-1}^{-1} (S') \setminus V(F_n)$ 
are locally closed submanifolds of $\Omega_n$ of type (I) or (II).  

Let $S$ be a connected component of  
$ \pi_{n,n-1}^{-1} (S') \setminus V(F_n)$ and  let $p\in S'$ be such that $p' = \pi_{n,n-1} (p)$.   
Then $\mathcal S_n$ near $p$ is the product of $\mathcal S_{n-1} \times \K$.  
Therefore the conclusion follows from the inductive assumption and the fact that $F_n(p)\ne 0$.  

If  $S$ is a connected component of $ \pi_{n,n-1}^{-1} (S') \cap V(F_n)$ 
we show that $\Ph'$ can be lifted to an {\aaa } trivialization  
$$\Ph :B\times \Omega'_0\times \K \to \Omega'\times \K,$$
 so that $\Ph$ preserves $\mathcal S_n$ and $F_n$ is $\Ph$-regular.  
 This can be done exactly as in the proof of Theorem \ref{theorem} as follows.  
 Denote by $a(y) = (a_1(y), ... , a_{d_n}(y))$ the vector of complex roots of $F_n$ and set 
 $$
 \Ph(y,x_n) = (\Ph' (y), \ps (x_n, a(0,y_{j+1},...,y_{n-1}), a(\Ph' (y))),
 $$
 where $\ps$ is given by the Whitney interpolation formula \eqref{ourpsi}.   
 The last claim of Proposition follows from Example \ref{derivationcomplete} and Theorem \ref{theoremtransverse}.  
 \end{proof}

\subsection{{$\Ph$ of the proof of Theorem \ref{theorem} satisfies condition (5) of Definition \ref{maindefinition2}.}}

Let ${\mathcal S},   \mathcal S_0$ be the canonical stratifications  associated to  the families 
$\{F_i(t,x)\}$ and $\{F_i(0,x)\}$ respectively.  We show that $\Ph$ induces a real analytic diffeomorphism 
between the strata of $B_\varepsilon \times \mathcal S_0$ and 
$ {\mathcal S}$.  By induction on $n$ we may suppose that the corresponding property holds for 
$\Ph_{n-1}$.  Let $S$ be a stratum of $ {\mathcal S}$ of type (I), that is a covering space over a stratum $S'$.  Denote  $S_0= S\cap \{t=0\}$, $S_0'= S'\cap \{t=0\}$.   By construction  $\Ph$ restricted to 
$B_\varepsilon\times V(F_n(0,x))$ is a lift of $\Ph_{n-1}$.  Therefore,  if $\Ph_{n-1}: B_\varepsilon\times S_0'\to S'$  is an analytic  diffeomorphism, consequently so  is its lift $\Ph: B_\varepsilon\times S_0 \to  S$.  

Now suppose that  $S$ is of type (II).  By assumption, $a_i(t,x')$ of \eqref{definitionPsi} are real analytic on 
$B_\varepsilon\times S_0''$ and hence, by the Whitney interpolation formula \eqref{ourpsi},  so is $\Ps_n$ on $B_\varepsilon\times S'_0$.  This shows the claim.  \qed

\begin{rem}
In general for  an (arc-a) or (arc-w) stratification, we have to substratify to obtain the condition (5) of Definition \ref{maindefinition2}.  For the canonical stratification  associated to a system of pseudopolynomials  the {\aaa } trivializations constructed in the proof of Proposition \ref{canonicalstratification} are real analytic on its strata.    
\end{rem}



{\Large \part{Applications.}\label{part2}}
\medskip


\section{Generic {\aaa } equisingularity} 

We use the Zariski equisingularity  to show that an analytic family of analytic set germs 
$\mathcal X = \{X_t\}$, $t\in T$, is  "generically" equisingular.  That is, locally on the parameter space $T$, this  family is equisingular in the complement of an analytic subset $Z\subset T$, $\dim Z < \dim T$.  
In this section the parameter space $T$ may be singular.  

\begin{defi} 
Let $T$ be a $\K$-analytic space, $U\subset \K^n$ an open neighborhood of the origin, 
$\pi : T\times U \to T$ the standard projection,  and let 
$\mathcal X=\{X_k\}$ be a finite family  of analytic subsets of $T\times U$.  We say that  $\mathcal X$ 
is \emph{{\aab }  equisingular along $T\times \{0\}$ at $t\in Reg(T)$}, if there are neighborhoods  $B $ 
of $t$ in $Reg(T)$ and $\Omega$ of  $(t,0)$ in $T\times \K^n$, and an {\aaa } trivialization 
$\Ph : B \times \Omega_t \to \Omega$, where $\Omega_t = \Omega\cap \pi^{-1} (t)$, such that 
$\Ph (B\times \{0\} ) =  B\times \{0\} $ and for every $k$, $\Ph (T\times X_{k,t}) = X_k, $ 
where $X_{k,t}= X_k \cap \pi^{-1} (t)$. 

We say that  $\mathcal X$ is  \emph{regularly {\aab }  equisingular along 
$T\times \{0\}$ at $t\in T$} if, moreover, $\Ph$ is regular at $(t,0)$.    
\end{defi}

\begin{thm}\label{genericequisingularity}
Let $\mathcal X= \{X_k\}$ be a finite family of analytic subsets of a neighborhood of $T\times U$ and 
let $t_0\in T$.   
Then there exist  an open  neighborhood $T' $ of $t_0$ in $T$ and a proper $\K$-analytic subset $Z\subset T'$, containing $Sing(T')$,  such that for every $t\in T'\setminus Z$,  $\mathcal X$ 
is  regularly {\aab }  equisingular along $T\times \{0\}$ at $t$.  

Moreover, there is an analytic stratification of an open neighborhood of $t_0$ in $T$ such that for every stratum $S$ and every $t\in S$,  $\mathcal X$ 
is  regularly {\aaa }  equisingular along $S\times \{0\}$ at $t$.
\end{thm} 

\begin{proof}
For each $X_k$ fix a finite system of generators $F_{k,i}\in \mathcal O_{T,t_0}$ of the ideal defining it. 
The first claim follows  from  Lemma \ref{theoremequisingularity2}   
applied to the product of all $F_{k,i}$.  The second claim follows by  induction on $\dim T$.  
  
  \begin{lem}\label{theoremequisingularity2}
Let $T$ be a $\K$-analytic space, $t_0\in T$.  Let $F$ be a $\K$-analytic function defined in 
a neighborhood of  $(t_0,0) \in T\times \K^n$.   Then there exist  a  neighborhood $T' $ of $t_0$ in $T$ 
and a proper $\K$-analytic subset $Z\subset T'$, $\dim Z < \dim T$, $Sing (T) \subset Z$, such that, after a linear change of coordinates in $\K^n$, the following holds.  For every $t\in T'\setminus Z$ there is a  Zariski equisingular  local transverse system of pseudopolynomials  $F_i$, $i=0,...,n$, at $(t,0)$,  with 
$F_n$ being the Weierstrass polynomial associated to $F$ at $(t,0)$. 
\end{lem}

\begin{proof}
We may suppose that $T$ is a subspace of $\K^m$, $t_0=0$, and $(T,0)$ is irreducible.  

We construct a new system of coordinates $x_1,... , x_n$ on $\K^n$, analytic subspaces $(Z_i,0) \subset (T,0)$ and analytic function germs $G_i (t,x_1,... ,x_i)$,  $i=n,n-1, ... ,0$, such that 
for every $t\in T\setminus Z$, $Z= Sing (T) \cup \bigcup Z_i$,   the following condition is satisfied.  Let $F_i $ be the  Weierstrass polynomial in $x_i$  associated to the germ of $G_i$ at $(t,0)$.  Then 
the discriminant of $F_{i,red}$ divides $F_{i-1}$.

The $G_i$ are constructed by descending induction.  First we set $G_n =F$.  Then we construct $G_{n-1} $ in three steps.  

\medskip
\noindent \emph{Step 1.}    Write 
$$
G_n (t,x) = \sum_{|\alpha|\ge m_0} A_\alpha (t) x^\alpha,
$$
where $m_0$ is the minimal integer $|\alpha|$ for which $A_\alpha \not \equiv 0$.  
We may assume $m_0>0$ otherwise we simply take  $Z= Sing (T)$. 
After a linear change of  $x$-coordinates, we may assume $A_{(0, \ldots, 0,m_0)}(t) \not \equiv 0$. Denote $A(t) =  A_{(0, \ldots, 0,m_0)}$. 
\medskip

\noindent \emph{Step 2.}   We define 
$A(t) \ast x := (A(t)^2 x_1 , .... , A(t)^2 x_{n-1}, A(t) x_n)$  and set
$$
\tilde G_n (t,x) = (A(t))^{- (m_0+1)}  G_n (t,A(t) \ast x) = \sum_{|\alpha|\ge m_0} \tilde A_\alpha (t) x^\alpha.  
$$
Then $\tilde A_{(0, \ldots, 0,m_0)}\equiv 1$ and 
 $\tilde G_n$ is regular in $x_n$.  
\medskip

\noindent \emph{Step 3.}
Denote by $H_n$ the Weierstrass  polynomial in $x_{n}$ 
associated to $\tilde G_n$.  It is  of degree $m_0$ in $x_n$.  
Let $\mathcal K$ be the field of fractions of  $\mathcal O_{T\times \K^{n-1} ,0}$ and consider $H_n$ as a polynomial of $\mathcal K[x_n]$.    Let $d$ be the degree of  $H_{n,red}$.   We define $G_{n-1}$ as  
the $d$th generalized discriminant of $H_n$, see Appendix \ref{Part:discriminants},  and set   $Z_n = A^{-1} (0)$. 

Then we repeat these steps for $G_{n-1}$ and so on.

To see that the sequence $G_i$ satisfies the required properties we note that if 
 $F_n$ denotes the Weierstrass polynomial at $(t,0)\in T\setminus (Z_n\cup Sing(T))$  
associated to $G_n$,  then, as a germ at $(t,0)$,  the discriminant of $F_{n,red}$ divides $G_{n-1}$.  
   
This ends the proof of Lemma \ref{theoremequisingularity2} and Theorem \ref{genericequisingularity}.  
\end{proof}
\renewcommand{\qedsymbol}{}
\end{proof}



\section{Stratifications and Whitney Fibering Conjecture.}\label{stratifications}

Let $X$ be a $\K$-analytic space of dimension $n$.  By \emph{an analytic stratification of $X$} we mean a filtration of $X$ by analytic subspaces 
$$X=X_n\supset X_{n-1} \supset \cdots \supset X_0$$
such that each $X_j\setminus X_{j-1}$ is a nonsingular (locally closed) analytic subspace of pure dimension $j$, or is empty.  This filtration induces a 
 decomposition $X=\sqcup S_i$, where the $S_i$ are connected components of all $X_j\setminus 
 X_{j-1}$.  The analytic locally closed submanifolds $S_i$ of $X$ are called \emph{strata} and their 
 collection $\mathcal S= \{S_i\}$ is usually called a stratification of $X$.  In what follows we simply say 
 that $\mathcal S= \{S_i\}$ is an analytic stratification of $X$, meaning that it comes from an analytic filtration.  Similarly we define an algebraic stratification of an algebraic variety.  

Stratifications are often considered with extra regularity conditions such as Whitney's conditions (a) and (b)  or  the (w) condition  of Verdier.  For more details and insight 
we refer the reader to \cite{whitneyannals}, \cite{whitney}, \cite{wall}, \cite{GWPL}, \cite{verdier}, \cite{GMac}, \cite{Gore} and the references therein.  
Recall that for a real analytic stratification the (w) condition implies the
 conditions (a) and (b), see  \cite{verdier}.  For a complex analytic stratification the conditions (w) and (b) are equivalent  \cite{teissier1982}.  
 
 We say that a stratification $\mathcal S= \{S_i\}$ is \emph{compatible with $Y\subset X$} if $Y$ is a union of strata.

\subsection{(Arc-a) and (arc-w) stratifications.}\label{arcstratifications}
Let $X$ be a $\K$-analytic space and let   $\mathcal S$ be an analytic  stratification of  $X$.  
Let $p$ be a point of a stratum $S\in \mathcal S$.  We say that  $\mathcal S$ is  \emph {{\aab } trivial   
at $p$},  or   \emph{satisfies the condition (arc-a) at $p$},  if  the following holds.  There are a neighborhood $\Omega$ of $p$, $\K$-analytic coordinates on $\Omega$ such  that 
 $ B=S\cap \Omega$ is a neighborhood of the origin in $ \K^m \times \{0\}$, and an  {\aaa } trivialization of  the projection $\pi$ 
  on the first $m$ coordinates 
\begin{align}\label{trivialization}
\Ph (t,x) : B \times \Omega_0 \to \Omega, 
\end{align}
where $\Omega_0 = \Omega\cap \pi^{-1} (0)$, such that $\Ph (B\times \{0\}) =B$ and $\Ph$ preserves the stratification.  By the last condition we mean that each stratum of $\mathcal S$ is the union of leaves of $\Ph$, see Section \ref{Aaa}.  We say, moreover, that $\mathcal S$ is  \emph {regularly  {\aab } trivial   
at $p$}, or   \emph{satisfies the condition (arc-w) at $p$}, if $\Ph$ of \eqref{trivialization} is regular along $B\times \{0\}$ in the sense of Definition \ref{maindefinition3}.  

We say that $\mathcal S$ is  \emph {{\aab } trivial}, or \emph{satisfies the condition (arc-a)}, if it 
does it  at every point of $X$. 
Similarly we define  \emph {regularly {\aab } trivial stratifications}.  
If $\mathcal S$ is regularly {\aab } trivial 
then, for short, we say that $\mathcal S$ \emph{satisfies the  condition (arc-w)}.  

We say that \emph{the condition  (arc-a), resp. (arc-w), is satisfied along a stratum $S$} if it is satisfied 
at every $p\in S$.  Similarly we say that \emph{the condition (a) or (w) is satisfied along a stratum $S$} 
if for every other stratum $S'$, $S\subset \overline S'$, the pair $S',S$ satisfies the respective  condition.

\begin{thm}\label{a}
If a stratification $\mathcal S$ satisfies the condition (arc-a), resp. the condition (arc-w), along a stratum $S$ 
then it satisfies  the condition  (a) of Whitney, resp. the condition  (w)  of 
Verdier along $S$.  
\end{thm}

\begin{proof}The first claim follows from  the continuity of the tangent spaces to the 
leaves of an {\aaa } trivialization, see Proposition \ref{continuitytangent}. 

We show the second claim.  Fix two strata $S_0 \subset \overline S_1$.  
If the  condition  (w) 
 fails for the pair $S_1, S_0$ at $p_0\in S_0$, then, by the curve selection lemma, it fails along a real analytic arc 
 $ p(s): [0,\varepsilon ) \to \overline S_1$  with $p_0=p(0)\in S_0$ and $p(s)\in S_1$ for $s>0$.  
We show that there is a $C^1$ submanifold $(M,\partial M) \subset (\overline S_1, S_0)$,  
 $\partial M=S_0$ near $p_0$, $p(s) \in M$, such that  $M\setminus \partial M, \partial M$ satisfies the condition  (w).  It then follows, see for instance \cite{denkowskawachta}, that the condition (w) is satisfied 
 along $p(s)$ for the pair $S_1, S_0$,  which contradicts the choice of $p(s)$ and hence completes the proof. 

We define $M$ using the trivialization $\Ph(t,x) = (t, \Ps (x))$ of \eqref{trivialization}.  
By the arc-analyticity of $\Ph^{-1}$ there is a real analytic arc $(t(s),x(s))$ such that 
$p(s) = \Ph (t(s),x(s))$.  Then we set 
\begin{align}\label{wing}
M = \{\Ph (t,x(s)); t\in B, s\ge 0\}.
\end{align} 
It follows from Proposition \ref{criterionregular} that $M$ is a $C^1$-manifold with boundary and that 
  $M\setminus \partial M, \partial M$ satisfies the condition  (w). 
 \end{proof}
 
 \begin{cor}[\cite{speder}] 
If a complex analytic hypersurface $X$ is generic Zariski equisingular along a nonsingular subspace
 $Y\subset Sing (X)$  then the pair 
 $Reg(X), Y$ satisfies Whitney's conditions (a) and (b).  
 \end{cor}
 
In the above proof of Theorem \ref{a} we use the Wing Lemma argument, the manifold $M$ being the wing.  
This method was introduced by Whitney in \cite{whitneyannals} and then  was used by many authors to show the existence of stratifications satisfying various regularity conditions, see for instance \cite{wall}, 
\cite{denkowskawachta}, \cite{BCR}. 
For a wing $(M,\partial M)$ the condition of being a $C^1$ submanifold is not sufficient to guarantee 
the condition (w)  for the pair $M\setminus \partial M, \partial M$, see  \cite{brodersentrotman} for 
examples.  Thus it is essential that the wing $M$ admit a parameterization \eqref{wing} satisfying 
\eqref{boundedder} of Proposition \ref{criterionregular}.  
Moreover, we show below the existence of a wing that admits a 
$\K$-analytic  parameterization and contains a given real analytic arc. 

\begin{prop}[Wing Lemma] \label{winglemma}
Let $\mathcal S$ be an (arc-a) stratification of a $\K$-analytic space $X$.  Let 
$ p(s): [0,\varepsilon ) \to X$  be a real analytic arc such that $p_0=p(0)\in S_0$ and $p(s)\in S_1$ for $s>0$ and a pair of strata $S_0, S_1$.  Then, there are  a local system of coordinates at $p_0$, 
$(X,p_0)\subset (\K^N,0)$,  an open neighborhood $\Omega$ of $p_0$ in $\K^N$ such that 
 $ B=S_0\cap \Omega$ is a neighborhood of $p_0$ in $ \K^m \times \{0\}$, a neighborhood $D$ of $0$ in 
 $\K$, and  $\K$-analytic maps 
 $$t(s): D \to B, \quad \varphi : B\times D \to X, \quad \varphi(t,s)=(t,\psi (t,s)),$$ such that  $\varphi (t,0)=(t,0) \in S_0$, 
 $p(s) = \varphi(t(s),s)$  for $s>0$,  and
 \begin{itemize}
 \item
 $\varphi (t,s) \in S_1$ for $s> 0$  if $\K=\R$ 
  \item
$\varphi (t,s) \in S_1$  for $s\ne 0$ if $\K=\C$. 
 \end{itemize}
Moreover if $\mathcal S$ is an (arc-w) stratification and if we write $\psi (t,s) = \sum _{k\ge k_0} D_k(t) s^k,$ then we may require that $D_{k_0} (0) \ne 0$.  
\end{prop}

\begin{proof}
The real case follows from the definition of an {\aaa } trivialization, and in the regular case from 
Proposition \ref{criterionregular}.  In the complex case we may construct  a complex wing 
as follows.  Let $\Ph(t,x) $ be the {\aaa } trivialization given in  \eqref{trivialization}   
 and let $p(s) = \Ph (t(s),x(s)): (I,0) \to (\overline S_1,p_0)$.  
 Then $ \Ph (t,x(s))$   as a power series defines a complex analytic map 
 $\varphi : (T\times \C, 0)   \to (\C^N, 0)$.  Thus $\varphi (t,s) = \Ph (t,x(s))$ for $s$ real, but not necessarily for  $s\in \C\setminus \R$.  Because $\Ph$ preserves the strata,  the stratum $S_1$  
  contains the image of $\varphi $ for $s> 0$. Since $\overline S_1$ is a complex analytic set, 
it contains the entire image of $\varphi $.   By assumption $\varphi (t,s) \notin S_0$ for $s\in \R, s>0$.  
Therefore, by Lemma \ref{complexification},  $\varphi (t,s) \notin S_0$ for $s\in \C\setminus \{0\}$ as claimed.
This ends the proof. \end{proof}


\subsection{Local Isotopy Lemma.}
Let $X$ be a Whitney  stratified  space, $p\in X$, and let $S$ be the stratum containing $p$.  
Then, as follows from Thom's first isotopy lemma, \cite{mather}, \cite{thom1969}, \cite{GWPL}, 
any local submersion onto $S$, restricted to $X$, can be trivialized over a neighborhood of $p$ in $S$.  As it follows from Proposition \ref{change} below an analogous property holds for (arc-a) and (arc-w) stratifications.  

Suppose that  $X$ is a $\K$-analytic subspace of a neighborhood of the origin in $\K^N$,  
 $\Omega$ a neighborhood of the origin in $X$,  and let $ B= \Omega \cap (\K^m \times \{0\})$.  
 Let $f$ be a $\K$-analytic function on $X$ and let $\mathcal X=\{X_k\}$ be a finite family of analytic subsets  
 of $X$.  Let $\pi: \Omega \to B$ denote the standard projection onto the first $m$ coordinates.

\begin{prop}\label{change}
Let $\Ph$  be   an  {\aaa } trivialization of  $\pi$ that preserves $B$ and a family of analytic subsets $\mathcal X=\{X_k\}$ of $X$ and let  $f$  be a  $\Ph$-regular function germ.  
Let $\tilde \pi$ be another analytic submersion $\Omega \to B$.  Then, after restricting 
to a smaller neighborhood of the origin,  there is an  {\aaa } trivialization $\tilde \Ph$ of  $\tilde \pi$, 
preserving $B$ and the  family $\mathcal X$, and such that $f$ is $\tilde \Ph$-regular.  
Moreover, if $\Ph$ is regular along $B$  then $\tilde \Ph$ can be chosen regular along $B$.
\end{prop}

\begin{proof}
Let $H: {B \times \Omega_0} \to {B \times \Omega_0}$ be given by 
$$
H(t,x) =(h(t,x),x) =  (\tilde \pi(\Ph (t,x)),x). 
$$
We show that $H$ is a local homeomorphism, {\aaa} in $t$, such that $H^{-1}$ is also {\aaa} in $t$.  
Then $\tilde \Ph= \Ph \circ H^{-1}$ satisfies the claim.  

Firstly $H$ is a local homomorphism by the implicit function theorem, Theorem 2.5 Ch. I  of \cite{hartman}.  
Let $\gamma:= x(s)$ be a real analytic arc.  Consider 
$$
H_\gamma(t,s) =(h(t,x(s)),s) :(B\times I,0) \to (B\times I,0). 
$$
$H_\gamma$ is  clearly $\K$-analytic in $t$ and real analytic in $s$.  
Since $h(t,x(s)) = t + \varphi (t,s)$ with $\varphi (t,s) \in \mathfrak{m}_{t,s}$,  $H_\gamma$ is 
a local analytic diffeomorphism and its inverse is $\K$-analytic in $t$ and real analytic in $s$.  
\end{proof}

Let $\mathcal S$ be an analytic  stratification of  $X$ satisfying Whitney's condition (a).  One says 
after  D\'efinition 4.1.1 of  \cite {BMM}, that $\mathcal S$ satisfies \emph{the stratified local triviality condition}, 
the  (TLS) condition for short,  if any local submersion onto a stratum is locally topologically trivial by a 
strata  preserving  trivialization.  Thus Lemma \ref{change} gives the following result.

 \begin{cor}\label{TLS}
 A stratification satisfying the condition (arc-a) also satisfies  the condition (TLS) of \cite {BMM}.  
 \end{cor}  
 
 \begin{proof}
 If we assume in Proposition \ref{change} the  $\tilde \pi$ is only a $C^1$-submersion then,  by 
Theorem 2.5 Ch. I  of \cite{hartman}, $H$ constructed  in the proof is a homeomorphism and so is $\tilde \Ph$. 
 \end{proof}


\subsection{Proof of Whitney fibering conjecture.}\label{proof}

We show below that every $\K$-analytic space admits locally an (arc-w) stratification.  In the algebraic case such stratification exists globally.  Since an (arc-w) stratification satisfies all the properties 
required by Whitney, it shows Whitney's fibering conjecture in the algebraic and local analytic cases.  

\begin{thm}\label{theoremconjecture}  
Let $\mathcal V = \{V_i\}$ be a finite family of analytic subsets of an open $U\subset \K^N$.  Let 
$p_0\in U$.  Then there exist an open neighborhood $U'$ of $p_0$ and an analytic stratification 
 of $U'$  compatible with each $U'\cap V_i$ and satisfying the condition (arc-w).   
\end{thm} 

\begin{proof}[First proof of Theorem \ref{theoremconjecture}] 
\changes{We construct a system of pseudopolynomials $F_i(x_1, \ldots ,x_n)$, see Section \ref{associatedstratification}, in a system of local coordinates at $p_0$, 
so that the canonical stratification associated to $\{F_i\}$ is compatible with $\mathcal V$.    
Since, by construction, this system will be derivation complete, the theorem follows from 
Proposition \ref{canonicalstratification}. }

\changes{
First for every analytic space $V_i$ choose a finite system of  generators of its ideal $I(V_i) = (g_{i,j})_{j=1,...n_i}$ 
in the local ring $\mathcal O_{p_0}$, and let $f_n=\prod_{i,j} g_{i,j}$.  After changing linearly the system of coordinates, if necessary,  we may assume that $f_n$ is regular in $x_n$ and then we replace it by the associated Weierstrass polynomial.  Let $F_n$ be the product of all (non-zero) partial derivatives $\partial/\partial x_n$ 
of $f_n$.  After a multiplication by a non-zero constant, we may assume that $F_n$ is monic in $x_n$.   }

\changes{
  Then define $f_{n-1}(x_1, \ldots ,x_{n-1})$ as the discriminant  of $f_{n,red}$ (or an appropriate higher 
  order discriminant, see Appendix \ref{Part:discriminants}).  After a linear change of coordinates $x_1,. . . ,x_{n-1}$ we may assume 
that   $f_{n-1}$ is regular in $x_{n-1}$ and then replace it by the associated Weierstrass polynomials.   
Let $F_{n-1}$ be the product of all (non-zero) partial derivatives $\partial/\partial x_{n-1}$ 
of $f_{n-1}$.  We continue this construction and thus define the system $F_i$.  If  it happens that $f_i$ is a non-zero constant  we define $F_i$ and all $F_j, j<i$, as identically equal to one.   Then, if $\varepsilon_i$ are chosen so that 
$0<\varepsilon_1 \ll  \cdots \ll \varepsilon_n\ll 1$,  this system satisfies the requirements 
of Definition \ref{systempseudopolynomials}.  }
\end{proof}

We now give a second proof of Theorem \ref{theoremconjecture}.  It is 
 less algorithmic but provides  a stratification with local transverse Zariski equisingularitiy.  This proof is based on the following lemma. 

\begin{lem}\label{lemequisingularity}
 Let  $F$ be  a $\K$-analytic function defined in a neighborhood of  $0\in  \K^N$ and let  $Y$  
 be a $\K$-analytic subset  of  a  neighborhood of  $0\in  \K^N$, $\dim Y =m$.   
Then there exist  a  neighborhood $U$ of  $0\in  \K^N$ and a  $\K$-analytic 
subset $Z\subset  Y\cap U$,  $\dim Z < m$,  $Sing(Y) \subset Z$, such that for every $p\in Y\cap U\setminus Z$,  
there are a local system of coordinates at $p$ in which $(Y,p)= (\K^m\times \{0\},0)$ and 
 a Zariski equisingular transverse system pseudopolynomials $F_i$, $i=0,...,n=N-m$, at $p$, such that  $F_n$
 is  the Weierstrass polynomial associated to $F$.  
\end{lem}

\begin{proof}
Choose a local system of coordinates at $p$ such that   the 
projection on the first $m$ coordinates restricted to $Y$ is finite.  Let 
$$
\varphi: Y\times \K^n \to \K^N, \qquad \varphi(y,x) = y + (0,x).  
$$
Then apply Lemma \ref{theoremequisingularity2} to $T=Y$ and $F(\varphi(t,x))$.   
\end{proof}

\begin{proof}[Second proof of Theorem \ref{theoremconjecture}] 
We construct a sequence of analytic set germs at $p_0$ $$U= X_n\supset X_{n-1} \supset \cdots \supset X_0 $$
whose representatives in a sufficiently small neighborhood $U'$ of $p_0$ define a stratification  satisfying 
the statement.  For simplicity of notation we assume $p_0$ to be the origin.  

First for each analytic space $V_i$ choose a finite system of  generators of its ideal $I(V_i) = (g_{i,j})_{j=1,...n_i}$ 
in the local ring $\mathcal O_{0}$, and let $f_{n}$ be the product of all of them: $f_{n}=\prod_{i,j} g_{i,j}$.  
In the first step we apply Lemma \ref{lemequisingularity} to $F=f_{n}$ and $Y=U$ and we set $X_{n-1} = Z$.   
If $\dim (X_{n-1},0)<n-1$ then we set $X_{n-2} = X_{n-1}$.  Otherwise we again apply Lemma \ref{lemequisingularity} to $F=f_{n}$ and $Y=X_{n-1}$ and we set $X_{n-2}$ equal to the obtained $ Z$.   

If $\dim (X_{n-2},0)<n-2$ then we set $X_{n-3} = X_{n-2}$.  Otherwise choose  a finite system of  generators $I(X_{n-1}) = (h_{n-1,j})$ and let $f_{n-1}= f_{n} \prod_j h_{n-1,j}$.  Next apply Lemma \ref{lemequisingularity} to $F=f_{n-1}$ and $Y=X_{n-2}$ and we set $X_{n-3}$ equal to the obtained $ Z$. 

The inductive step is then the following.  Given $U= X_n\supset X_{n-1} \supset \cdots \supset X_{i} $ and a function $f_{i+1}$ that is the product of $f_{i+2}$ and a finite set of generators of $I(X_{i+1})$ in $\mathcal O_0$.  
If $\dim (X_{i},0)<i$ then we set $X_{i-1} = X_{i}$.   Otherwise we apply Lemma \ref{lemequisingularity} to $F=f_{i+1}$ and $Y=X_{i}$ and we set $X_{i-1}$ equal to the obtained $ Z$. 

Let $p\in X_k\setminus X_{k-1}$.  Then by construction 
there is a local system of coordinates at $p$ in which $(X_k,p) = (\K^k,0)$ and an {\aaa } trivialization $\Ph$ 
 of the coordinate projection on $\K^k$, preserving $X_k$ and such that $f_{k+1}$ is $\Ph$-regular.  Therefore, 
 $\Ph$ preserves the zero set of every factor of $f_{k+1}$ and hence every $V_i$ and every $X_j$ for $j>k$.  
 This ends the proof.  
\end{proof}



\subsection{Remark on Whitney fibering conjecture in the complex case} \label{remark}

Let  $U$ be a neighborhood of $0\in \C^m\times \C^n$.  Set $M= U\cap (\C^m\times \{0\})$ and 
$N= \{0\} \times \C ^n$.  Suppose, following Whitney,  that there exists a  homeomorphism  
$$\phi (p,q) : M\times N \to U,$$ 
complex analytic in $p$, such that $\phi (p,0)= p$ and $\phi (0,q) =q$, and that 
for each $q\in N$ fixed, $\phi (\cdot ,q) : M\times \{q\}\to U$ is a complex analytic 
 embedding onto an analytic 
submanifold $L(q)$.  Now we make an additional assumption:

\bigskip
(A) \emph{ for all $q\in N$, $L(q)$ is transverse to $N$.} 

\bigskip
By continuity of  $\phi (p,q)$ we may assume that  the projection of $L(q)$ onto $M$ is proper.  
Therefore, by (A) and the assumption $\phi (0,q) =q$, is has to be of degree $1$.  Therefore  $L(q)$ is the graph of a complex analytic function $f_q: M \to \C^n$.   
If $q\to 0$ then the values of $f_q$ go to $0$ and hence, by Cauchy integral formula, 
 the partial derivatives of $f_q$ go to $0$ on relatively compact subsets of $M$.     This 
 ensures the continuity of the tangent spaces to the leaves $L(q)$ as $q\to 0$.

\subsection{Examples} \label{examples}

There are several classical examples describing the relation between the
Zariski equisingularity and Whitney's conditions that we recall below.  The general set-up 
for these examples is the following.  Consider a complex algebraic hypersurface 
$X\subset \C^4$ defined by a polynomial $F(x,y,z,t)=0$ such that $Sing X =T$, where $T$ 
is the $t$-axis. Let $\pi : \C^4 \to T$ be the standard projection. 
 In all these examples $X_t =\pi ^{-1} (t)$, $t\in T$, is a family of 
  isolated singularities,  topologically trivial along $T$.  These examples relate the following conditions :
 \begin{enumerate}
 \item
 $X$ is Zariski equisingular along $T$, i.e. there is a local system of coordinates in which $F$ can be completed to a Zariski equisingular system of polynomials, see Definition  \ref{system}.  
 \item 
  $X$ is Zariski equisingular along $T$ for a transverse coordinate system, 
   i.e. there is a local system of coordinates in which $F$ can be completed to a Zariski equisingular transverse system of polynomials, Section \ref{transversedefinition}.  
  \item
    $X$ is Zariski equisingular along $T$ for a generic system of coordinates,   i.e. for generic system of  local coordinates, $F$ can be completed to a Zariski equisingular  system of polynomials
    \item
    The pair $(X\setminus T,T)$ satisfies Whitney's conditions (a) and (b). 
  \end{enumerate}
 
Clearly (3)$\Rightarrow$(2)$\Rightarrow$(1).  Speder showed  (3)$\Rightarrow$(4) in \cite{speder}  
and (2)$\Rightarrow$(4) 
for families of complex analytic hypersurfaces with isolated singularities in $\C^3$  in his thesis 
\cite{spederthese} (not published). 
 Theorem \ref{a} gives (2)$\Rightarrow$(4) in the general case.   As the examples below show, all the other 
 implications are false.  
 
\begin{example}[\cite {brianconspeder1975b}]
\begin{align}
F(x,y,z,t)= z^5 + t y^6 z + y^7 x + x^{15} 
\end{align}
This example satisfies (1) for the projections  $(x,y,z) \to (y,z) \to x$ but  (4) fails.  
As follows from  Theorem \ref{a}, (2) fails as well.  
\end{example}

\begin{example}[\cite {brianconspeder1975a}]
\begin{align}
F(x,y,z,t)= z^3 + t x^4 z + y^6 + x^{6} 
\end{align}
In this example (4) is satisfied and (3) fails.  
This example satisfies (1) for the projections  $(x,y,z) \to (x,z) \to x$.   
\end{example}

\begin{example}[\cite {luengo}]
\begin{align}
F(x,y,z,t)= z^{16} + ty z^3 x^7 + y^6z^4+ y^{10} + x^{10} 
\end{align}
In this example (2) is satisfied and (3) fails.  
\end{example}

\begin{example}[\cite {parusinski1985}]
\begin{align}
F(x,y,z,t)= x^9 + y ^{12} + z^{15} + tx^3 y ^4 z^5 
\end{align}
In this example (4) is satisfied and (1) fails.  This shows also that (4) does not imply (2).
\end{example}


\section{Equisingularity of functions}\label{functions}

In this section we show how to use  Zariski's equisingularity  to obtain local topological triviality of 
analytic function germs.   We develop several different approaches.  

Firstly we show that the assumptions of Theorem \ref{theorem} gives not only 
the topological equisingularity of  sets, but also of the function $F_n$ 
and of any analytic function dividing $F_n$.  
To prove it we modify the vector fields defined by the {\aaa } trivialization $\Ph$,  so
that their flows trivialize $F_n$.  Note that this new trivialization is no longer {\aaa}.   

Then, for an analytic function $f$, we introduce new stratifying conditions  (arc-a$_f$) and (arc-w$_f$), 
analogs of  conditions (arc-a) and (arc-w), and show that they imply 
the classical stratifying conditions  (a$_f$) and (w$_f$) respectively.  
 
Finally we show how to adapt the Zariski equisingularity to the graph of a function $f$ 
in order to obtain an  {\aaa } triviality of $f$.    


\subsection{Zariski equisingularity implies topological triviality of the defining function.} \label{Zariskigivesf} 
We show that the assumptions of Theorem \ref{theorem} give  not only the topological triviality of the 
zero set of $F_n$ but also of $F_n$ as a function.  

\begin{thm}\label{a_ftoptivis}
Let  $B$, $\Omega_0$ and $\Omega$ be neighborhoods of 
the origin in $\K^m$, $\K^n$, and $\K^{m+n}$ respectively,  and let $\Ph  : B \times \Omega_0 \to \Omega$ be an arc-wise analytic trivialization satisfying the condition (Z1) of Theorem \ref{theorem}.  
Let $f(t,x)$ be a $\K$-analytic  $\Ph$-regular function.  Then $f$ is topologically trivial along 
$ B \times\{0\}$ at the origin, i.e. 
there are smaller neighborhoods $B'$, $\Omega'_0$ and $\Omega'$  
and a homeomorphism 
$$
h  : B' \times \Omega' _0 \to \Omega' 
$$
such that 
$h(t,0) = (t,0)$, $ h(0,x) = (0,x)$,  
and $f(h(t,x))= f(0,x)$.  
\end{thm}

\begin{proof}
The trivialization $h$ is obtained by  integrating the vector fields $w_i(t,x)$  defined below.  
Let $v_i$  be the vector fields on $\Omega$  given by   \eqref{vectorfields}.   
The regularity condition \eqref{boundderfunction} gives  
\begin{align}\label{boundforderG2}
|\partial f / \partial v_i|  \le C    |f|.
\end{align} 
Note that locally on $\Omega \setminus V(f)$ we may approximate $v_i$ by a $C^\infty$  vector field 
satisfying \eqref{boundforderG2}.  Using  partition of unity, we may glue such local approximations to 
a smooth vector field that satisfies \eqref{boundforderG2} and extends continuously to $V(f)$ by 
${v_i}|_{V(f)}$.  In what follows we replace $v_i$ by such approximation.  
 
Next we consider on $\Omega \setminus V(f)$ the orthogonal projection of $v_i(t,x) $ on the levels of $f$
$$
w_i(t,x)  = v_i(t,x) - \frac {\partial f / \partial v_i } {\| \grad f \| ^2} \grad f .
$$
(in the complex case  $\grad f:= \overline{(\partial f/\partial z_1, ..., \partial f/ \partial z_{m+n})}$ so that 
$\partial f/\partial v=\langle v, \grad f\rangle$).  Then we extend $w_i$ by  $v_i(t,x) $ onto $\Omega$.  
Clearly $\partial f / \partial w_i=0$ and $w_i$ are continuous by  \eqref{boundforderG2} 
and  {\L}ojasiewicz Gradient Inequality, \cite{lojasiewicz}, that says that there are constants $C>0$,  $\theta <1$, such that 
$$
\| \grad f \|  \ge C |f|^\theta .
$$
in a neighborhood of the origin.  By Remark \ref{uniquesolution}  the integral 
curves of $w_i|_{V(f)}$ are unique and hence they are unique on $\Omega$.  Therefore, by Theorem 2.1 of \cite{hartman}, for each $i$ the flow  $h_i$ 
of $w_i$ is continuous.  Then we trivialize $f$ by composing these flows:   
$$
h(t_1,...,t_m,x) = h_1(t_1,h_2( t_2,h_3(...(t_{m-1},h_m(t_m,x))...))).
$$
\end{proof}



\subsection{Conditions  (arc-a$_f$) and (arc-w$_f$).}\label{arcfstratifications}
Let $f:X\to \K$ be a $\K$-analytic function defined on a $\K$-analytic space $X$.  By 
\emph{a stratification of $f$}
 we mean an analytic  stratification $\mathcal S$ of $X$ such that $V(f)$ is a union of strata.  We also assume that for any stratum $S\subset X\setminus V(f)$,   $f|_S$ has no critical points.  
A stratification  $\mathcal S$ of $f$ is called \emph{a Thom stratification of $f$} if it is a Whitney stratification of $X$ that for each pair of strata satisfies Thom's condition ($a_f$).  
For a definition of condition ($a_f$) we refer the reader to \cite{thom1969}, \cite{mather}, 
 \cite{GWPL}, \cite{leteissier}, \cite{hammle}.  
For the strict Thom condition ($w_f$) see \cite{leteissier}  and \cite{HMS}.  

 We say that  a stratification $\mathcal S$ of $f$   \emph{satisfies the condition (arc-a$_f$) at $p\in V(f)$}  if  there exists a local {\aaa } trivialization 
 \eqref{trivialization} at $p$ preserving the strata of $\mathcal S$ and  such that $f$ is $\Ph$-regular  at $p$.   If, moreover, $\Ph$ is regular at $p$ then we say that $\mathcal S$ 
 \emph{satisfies the condition (arc-w$_f$) at $p$}.  
 
We say that \emph{the condition (arc-a$_f$), resp. (arc-w$_f$), is satisfied along  a stratum $S$} if it is satisfied 
at every $p\in S$.  Similarly we say that \emph{the condition ($a_f$) or  ($w_f$) is satisfied along  $S$} 
if for every other stratum $S'$, $S\subset \overline S'$, the pair $S',S$ satisfies the respective condition.
 
 We note that by the assumption that   $f|_S$ has no critical points on stratum $S\subset X\setminus V(f)$, 
 the levels of $f$ are transverse to $S$.  Therefore, if moreover $\mathcal S$ satisfies Whitney's condition (a), the conditions ($a_f$) and  ($w_f$) are automatically satisfied along such $S$.

\begin{thm}\label{a_f}
If a stratification of $f$ satisfies the condition (arc-$a_f$), resp.  (arc-$w_f$),    along  a stratum 
$S\subset V(f)$ then it satisfies  the Thom condition ($a_f$), resp. ($w_f$),  along $S$. 
\end{thm}

\begin{proof}
Similarly to the proof of theorem \ref{a} it suffices we check the conditions 
 along a real analytic curve by considering a wing containing the curve.  
 
Thus fix two strata $S_0 \subset \overline S_1$, $S_0\subset V(f)$, $S_1\cap V(f) = \emptyset$,  
and a real analytic curve $\gamma : p(s) = \Ph (t(s),x(s)): [0,\varepsilon ) \to \overline S_1$  with $p_0=p(0)\in S_0$ and $p(s)\in S_1$ for $s>0$.  First we consider the case $\K=\R$ and the wing 
$$M = \{\Ph (t,x(s)); t\in B, s\ge 0\}. $$  
By the regularity of $f$ for $\Ph$, Proposition \ref{criterionregularfunction}, we may reparametrize 
$\Ph _ \gamma (t,s) = \Ph (t,x(s))$  by replacing $s=s(t,\tilde s)$ so that $f(\Ph_\gamma (t, \tilde s))= 
\tilde s^{k_0}$.   If we write $\Ph_\gamma (t, \tilde s))= (t, \Ps_\gamma (t, \tilde s)))$ 
then the tangent space to the levels of $f_{|M}$ is generated by 
$D\Ph_\gamma ( \partial/\partial t_i , \partial \Ps_\gamma /\partial t_i)$, that tends to 
$( \partial/\partial t_i , 0)$ as $s\to 0$, $i=1,... ,m$.  This shows ($a_f$).  If moreover 
$\Ph$ is regular then the condition  ($w_f$) follows form Proposition \ref{criterionregular}.  

  If $\K=\C$, then we use the complex wing  of Proposition \ref{winglemma}.  
\end{proof}

\begin{cor}\label{w_ftriv}
Let $f:X\to \K$ be $\K$-analytic and let  $\mathcal S$ be a Whitney stratification of $f$ 
 satisfying the condition (arc-$a_f$).  Then $f$ is  topologically trivial along each stratum $S\subset V(f)$.  
\end{cor}

In the complex analytic case it is shown in  \cite {BMM} that any stratification of $f$ satisfying the conditions (a) and (TLS) also  satisfies the condition (a$_f$).    Similarly, after   \cite {BMM} and 
\cite{parusinski1993} any Whitney stratification of  $f$ satisfies the strong Thom condition ($w_f$).  
Analogous results are false in the real analytic case.  
Thus in the complex case Theorem \ref{a_f}  (for the stratification and not for a single stratum)  follows from  Theorem \ref{a}, Proposition \ref{complexregularity}, and Corollary \ref{TLS}. 

Thom's condition (a$_f$)  implies the  topological triviality of $f$ along the strata of a Whitney stratification.  
But the condition (a$_f$) alone  does not imply Whitney's condition  (b) and therefore it 
may not imply topological triviality of $f$ along  the strata.  
Similarly,  the condition (arc-a$_f$) alone   
may not itself imply topological triviality of $f$ along the strata.  Nevertheless, in some special cases, the topological triviality can be obtained by adapting the proof of Theorem \ref{a_ftoptivis}.  

\begin{cor}\label{a_ftriv}
Let $f:X\to \K$ be $\K$-analytic and let  $\mathcal S$ be a stratification of $f$ 
such that $X\setminus V(f)$ is a stratum of $\mathcal S$.  Suppose that $\mathcal S$ 
satisfies  the condition (arc-$a_f$) along a stratum $S\subset V(f)$.  Then $f$ is  topologically trivial along $S$.  
\end{cor}


\subsection{{\Aaa }  triviality of function germs.} \label{trivialityfunctions}

Consider a family of function germs  $f_t (y)=f(t,y): T\times (\K^{n-1},0) \to \K$, parametrized by an open  
$T \subset \K^m$.   
We say that $f_t$ is \emph{{\aab} trivial along $T$} if there are 
neighborhoods $\Lambda$ of $T\times \{0\}$ in $\K^m \times \K^{n-1}$ and   $\Lambda_0$ of $ \{0\}$ in 
$ \K^{n-1}$, $f_0: \Lambda_0\to \K$, and an {\aaa } analytic trivialization  
\begin{align*}
\sigma:  T \times \Lambda_0 \to \Lambda ,  \text { such that }  f(\sigma (t,y)) = f_0 (y) . 
\end{align*}

Using the method developed in  \cite{BPR} we have the following result. 

\begin{thm}\label{genericfunctionequisingularity}
 Let  $f_t (y)=f(t,y): T\times (\K^{n-1},0) \to \K$ be  a $\K$-analytic family of $\K$-analytic function germs 
 and let $t_0\in T$.    Then, there exist  a  neighborhood $U$ of  $t_0$ in $T$ and a  $\K$-analytic 
subset $Z\subset  U$, $\dim Z<\dim T$, such that $f$ is {\aab } trivial along $U\setminus Z$.  
\end{thm}

\begin{proof}
Set $x= (x_1, ... , x_n)=(x_1, y)$, where $y=(y_1,... ,y_{n-1})$, and  $F(t,x_1,y)= x_1- f(t,y)$.  
By Lemma \ref{theoremequisingularity2} and Theorem \ref{theorem}, together with Proposition \ref{addendum}, there is an {\aaa } trivialization $\Ph$ of the zero set of $F$ that preserves the levels of $x_1$.     Then 
$$
\sigma (t,y) = \pi (\Ph (t, f_0 (y),y) ) ,
$$
where $\pi$ is the projection $\pi (t,x_1,y) = (t, y)$,  gives an {\aaa } trivialization of $f$.  
\end{proof}


\section{Algebraic case}\label{algebraiccase}

\subsection{Construction of  {\aab } trivial stratifications }\label{construction}
Given a polynomial $F\in \K[x_1,...,x_n]$ we may construct a system of polynomials $F_{i} \in \K[x_1, \ldots, x_i]$, $i=1,...,n$, as follows.  First  we set $F_n = F$ that after a linear change of coordinates we may assume monic in $x_n$. Then let $F_{n-1}$ be the discriminant of $F_{n,red}$ or an appropriate higher order discriminant, see Appendix \ref{Part:discriminants}.  We again make a linear change of coordinates $x_1,...,x_{n-1}$ so that we may assume $F_{n-1}$ monic in $x_{n-1}$ and 
we continue until we get $F_j$ a non-zero constant. This construction is algorithmic except taking generic system of coordinates.  For such a system of polynomials $F_{i} \in \K[x_1, \ldots, x_i]$, $i=1,...,n$,  we 
may consider the canonical stratification defined in Section \ref{associatedstratification}.   
Moreover, we may refine this construction to obtain  a derivation complete system of polynomials.  
Then Proposition \ref{canonicalstratification} gives the following.

\begin{thm}\label{canonicalstratification3}
Given $F\in \K[x_1, \ldots ,  x_n]$.  There exists a linear system of coordinates $x_1, . . .  ,x_n$ on $\K^n$ and a 
derivation complete system of polynomials on $\{F_i(x_1, . . .  ,x_i)\}$ such that $F$ divides $F_n$.  In particular the associated canonical stratification to this system satisfies the condition (arc-$w$)  and 
the condition (arc-$w_f$)  for  any factor of $F$.  
\end{thm} 

Theorem \ref{canonicalstratification3} gives Whitney's Fibering Conjecture in the affine algebraic case.  
Since the above constructions preserves the family of homogeneous polynomials we obtain as well 
an algorithmic proof of the following projective algebraic version of Whitney's Fibering Conjecture.  

\begin{thm}\label{algebraictheoremconjecture}  
Let $\mathcal V = \{V_i\}$ be a finite family of algebraic subsets of  $\proj _\K^n$.    Then there exists an 
algebraic stratification of $\proj _\K^n$  compatible with each $ V_i$ and satisfying the condition (arc-w).  
  Moreover, 
the local {\aaa } trivializations can be chosen semi-algebraic.  
\end{thm}

Different proof of Theorems \ref{canonicalstratification3}, \ref{algebraictheoremconjecture}, that gives local {\aaa} trivialization by 
Zariski equisingular local transverse system of polynomials follows from Lemma \ref{algebraiclem2}.


\subsection{Generic {\aaa } equisingularity in the algebraic case.}

In the algebraic case we have a global version of Theorem \ref{genericequisingularity}.  
Here by a \emph{real algebraic variety} we mean an affine real algebraic variety in the sense of Bochnak-Coste-Roy \cite{BCR}:  a topological space with a sheaf of real-valued functions  isomorphic to a real algebraic set $X\subset \R^N$ with the Zariski topology and the structure sheaf of regular rational functions.  
For instance,  the set of real points of a reduced projective scheme over $\R$,  with the sheaf of regular functions, is a real algebraic variety in this sense.  

\begin{thm}\label{theoremequisingularity3}
Let $T$ be an algebraic variety (over $\K$)  and let $\mathcal X =\{X_k\}$ be  a finite family of 
algebraic subsets $T \times \proj^{n-1}_\K$.  Then there exists 
an algebraic 
stratification  $\mathcal S$ of $T$ such that for every stratum $S$ and for 
every $t_0\in S$  there is a neighborhood $U$ of $t_0$ in $S$ and a semialgebraic {\aaa } 
 trivialization  of $\pi,$ preserving the family $\mathcal X,$
\begin{align}\label{raagtrivial}
\Ph  : U \times \proj^{n-1}_\K \to \pi^{-1} (U),\end{align}
$\Ph(t_0,x)= (t_0,x)$, 
where $\pi:T \times \proj^{n-1}_\K\to T$  denotes the projection.   
\end{thm}

\begin{proof}

We may assume that $T$ is affine irreducible.    
Let $G_{i}(t,x)$, $t\in T$, $x=(x_1, \ldots ,x_n)$,  be a finite family of polynomials, homogeneous in $x$, defining the sets $X_k $ and let  $F_n(t,x)$ be the product of all $G_i$.  We consider $F_n$ as a homogeneous polynomial over $\mathcal K=\K(T)$ and let  
$$
F_n (x) = \sum_{|\alpha|=d_n} A_\alpha  x^\alpha, \quad  A_\alpha \in \mathcal K.
$$
After a linear change of coordinates $x$, we may suppose $A_n = A_{(0,\ldots,0,d_n)} \ne 0$.  
Then we define $F_{n-1}(x_1,... ,x_{n-1})$ as  the discriminant of 
 $F_{n,red}$, and proceed inductively by constructing the system of homogeneous 
 polynomials $F_j\in \mathcal K[x_1, ..., x_i]$ until $F_i \in \mathcal K$ is a non-zero constant.  
 Then we take as $Z\subset T$ the union  
 of  zero sets of the denominators of  the coefficients of all $F_j$ and the numerators of the leading coefficients of  all $F_j$.  
We show below  that the statement of theorem holds for $T\setminus Z$ as an open stratum.  Then the stratification $\mathcal S$ can be constructed by induction on $\dim T$.  
 
 By Theorem \ref{theorem},  $V(F_n)$ is  {\aab } equisingular along $(T\setminus Z)\times \{0\}$.  
 By construction  \eqref{definitionPsi}, the trivialisation $\Ph (t,x) = (t, \Ps (t,x))$  is semi-algebraic and 
  $\K^*$-equivariant in the variable $x,$ as follows from the interpolation formula, see 
 Remark \ref{symmetric}.  Moreover, by construction, it is regular along $U\times\{0\}$, $U$ being a neighborhood of $t_0$ in $T\setminus Z$. Then the  trivialization 
   $U \times \proj^{n-1}_\K  \to \pi^{-1} (U)$ induced  by $\Ph$, is {\aaa }.  
\end{proof}




We have the following  versions of Lemmas \ref{theoremequisingularity2} and \ref{lemequisingularity}. 

  \begin{lem}\label{algebraiclem1}
Let $T$ be a $\K$ algebraic variety and let $F\in \K[T\times \K^n]$, $F\not \equiv 0$.   Then there exists  a   subvariety $Z\subset T$, $\dim Z < \dim T$, such that, after a linear change of coordinates in $\K^n$, 
$F$ can be completed to a system of polynomials $\{F_i\}$, $F_n=F$, such that for every 
$t\in T\setminus Z$ the system $\{F_i\}$  is transverse and   Zariski equisingular   
 at $(t,0)$.  \qed
\end{lem}

\begin{lem}\label{algebraiclem2}
 Let  $F\in \K[X_1, ... ,X_N]$, $F\not \equiv 0$,  and  let  $Y\subset  \K^N$ be an algebraic subset.   
Then there exist  an algebraic $Z\subset  Y$,  $\dim Z < \dim Y$, and polynomials $\{F_i\}$, $F_n=F$, such that the following holds.  For every $p\in Y\setminus Z$ there is a local system of coordinates at 
$p$ in which $(Y,p)= (\K^m\times \{0\},0)$, such that the germs of  $\{F_i\}$ at $p$ form a 
 transverse and   Zariski equisingular  system of polynomials.  
  \qed
\end{lem}


\subsection{Applications to real algebraic geometry}\label{raag}
Let $X$ be a compact (projective or affine) real algebraic variety in the sense of \cite{BCR}.  
A functorial filtration on the semi-algebraic chains $C_*(X;\Z_2)$ was introduced in 
\cite{mccroryparusinski1}.  This filtration, called 
\emph{the Nash filtrtation},  
defines a spectral sequence, \emph{the weight spectral sequence of $X$}, that, in turn, defines 
\emph{the weight filtration} on the homology $H_* (X;\Z_2)$.  
This construction can be extended to non-compact real algebraic varieties and 
the Borel-Moore homology.  For a real algebraic variety $X$ 
its virtual Poincar\'e polynomial $\beta (X) \in \Z[t]$, introduced in \cite{mccroryparusinski}, is a multiplicative and additive 
invariant, an analog of the Hodge-Deligne polynomial.  As shown in \cite{mccroryparusinski1}, the virtual Poincar\'e polynomial can be computed from the weight spectral sequence.  
For the cohomological counterpart of this theory see \cite{limogespriziac}.  

The Nash filtration is functorial not only for regular morphisms but also for the $\AS$-maps that 
can be defined as follows.  Let $X, Y$ be compact real algebraic varieties.   A continuous map 
$f: X\to Y$ is \emph{an $\AS$-map} if its graph  $ \Gamma_f$ is a semialgebraic and arc-symmetric 
subset of $X\times Y$.   For instance a map that is semialgebraic and arc-analytic is 
$\AS$.  For more on $\AS$ maps see \cite{parusinski2004}, \cite{kurdykaparusinski}.  

Let $\Ph$ be a semialgebraic {\aaa } trivialization \eqref{raagtrivial} preserving real algebraic 
$X\subset T \times \proj^{n-1}_\K$ and let $X_t = \pi ^{-1}(t)$.  Then for each 
$t\in U$, $\Ph$ induces a semialgebraic and arc-analytic homeomorphism 
$$
\varphi_{t_0,t}: X_{t_0} \to X_t,
$$
with an arc-analytic inverse.  In particular, each $\varphi_{t_0,t}$ is $\AS$.  Thus Theorem 
\ref{theoremequisingularity3} gives the following. 

\begin{cor}\label{corollaryweight}
Let $T$ be a real algebraic variety and let $X$ be  an algebraic subset of  
$T \times \proj^{n-1}_\K$.  Then there exists an algebraic stratification  $\mathcal S$ of $T$ such that for every stratum $S$ and for every $t_0, t_1 \in S$  the fibres  $X_{t_0}$ and $X_{t_1}$ are $\AS$-homeomorphic and hence have isomorphic weight spectral sequences and weight filtration on the homology with $\Z_2$ coefficients. 
\qed
\end{cor}

\begin{cor}\label{corollaryvirtualbetti}
Let $T$ be a real algebraic variety and let $X$ be  an algebraic subset of  
$T \times \proj^{n-1}_\K$.  Then there exists an algebraic stratification  $\mathcal S$ of $T$ such that for every stratum $S$ the virtual Poincar\'e polynomial  $\beta (X_t)$ is independent of $t\in S$.   \qed
\end{cor}

The latter result was also shown in \cite{fichoushiota} by means of the resolution of singularities.  
\bigskip




\appendix
\renewcommand{\thesection}{\Roman{section}}
\section {Whitney Interpolation.}\label{Sec:interpolation}

We generalize the classical Whitney Interpolation formula \cite{whitney}, \cite{hardtsullivan}.

Fix positive integers $d,N$ and consider a family of functions $\f_i : \C^N \to \C$, $i=1, 2, ..., N$.   
We  assume that, for a constant $C>1$,   this family satisfies  the following properties

\begin{enumerate}
\item 
 $\f_i$ are continuous, differentiable on $({\C^*})^N$, and satisfies $f_i(\lambda \xi)= |\lambda|^d f_i(\xi)$ 
 for all $\lambda \in \C$.  
\item    
for every permutation $\sigma \in S_N$: $\f_i(\xi_{\sigma(1)}, ..., \xi_{\sigma(N)}) = \f_{\sigma (i)} 
(\xi_1,... , \xi_N)$.  
\item 
$|\f_j(\xi_1, \ldots ,\xi_N)| \le C |\xi_j| ( max _i |\xi_i|)^{d-1}$. 
\item  
for all $k,j$,  $| \xi_k^2| ( | \partial \f_j/ \partial \xi_k| + | \partial \f_j/ \partial \bar\xi_k| ) \le C  |\xi_j| ( max _i |\xi_i|)^{d}$.
\item  
 $\f=\sum_i \f_i$ is real valued and satisfies
$C^{-1}( max _i |\xi_i|)^{d} \le \f (\xi_1, \ldots ,\xi_N) \le C  ( max _i |\xi_i|)^{d}$.  
\end{enumerate}   

For examples see Examples \ref{exampleinterpol1} and \ref{exampleinterpol2}.  


 Given two subsets $\{a_1, \ldots, a_N\}\subset \C$, $\{b_1, \ldots, b_N\}\subset \C$,    of cardinality $N$ 
 such  that if $a_i=a_j$ then $b_i=b_j$.  
Define $ D_i = b_i - a_i$    and set 
\begin{align}
\gamma = \max_{a_i\ne a_j} \frac { |D_i - D_j|}{|a_i -a_j|}.  
\end{align}
Then 
\begin{align}\label{Dbound}
|D_i - D_j| \le \gamma |a_i -a_j|.  
\end{align}
Let 
$$\mu_i(z) :=\f _i((z-a_1)^{-1} , \ldots, (z-a_N)^{-1} )  , \quad \mu (z) 
:=\f ((z-a_1)^{-1} , \ldots, (z-a_N)^{-1} )  .$$
Define the  interpolation map $\ps: \C \to \C$ by 
\begin{align}\label{psireal} 
\ps (z) = z + \frac {\sum_{i=1}^N  \mu_i(z)  D_i}
 {\mu(z)} ,
\end{align}
if $z\notin \{a_1, \ldots, a_N\}$ and $\ps (a_i)=b_i$.  Then $\ps$ is continuous as follows from the following 
lemma.

\begin{lem}
\begin{align*}
\lim _{z\to a_j}\ps (z)=b_j.
\end{align*}
\end{lem} 

\begin{proof}
Let 
$I_j = \{i; a_i=a_j\}$.
We rewrite the interpolation formula \eqref{psireal} as 
\begin{align}\label{psi-j} 
\ps (z) = z + D_j  + \frac {\sum_{i\notin I_j} \mu_i (z) ( D_i- D_j)}   {\mu(z) } . 
\end{align}  
By  the properties (3) and (5),  for $i\notin I_j$, $\frac {\mu _i(z)}  {\mu (z)} \to 0$ as $z\to a_j$.   
\end{proof}


\begin{rem} \label{symmetric} \emph {Symmetries.} \\
The map $\ps$ is also invariant under permutations $\sigma \in S_N$, 
$\sigma (a) = (a_{\sigma (1)}, \ldots ,a_{\sigma (N)})$
\begin{align*}
\ps  (z,\sigma (a),\sigma (b)) = \ps (z,a,b) .  
\end{align*}
Let $\tau: \C \to \C$ be complex affine, $\tau (z) = \alpha z + \beta$.  Then 
\begin{align*}
\ps  (\tau (z) , \tau (a), \tau (b)) = \tau (\ps (z,a,b)) .  
\end{align*}
\end{rem}
 
 

  \begin{prop} \label {lipschitz}
The map $\ps :\C \to \C $ is Lipschitz with Lipschitz constant $4 N^3 C^4 \gamma  +1 $.  
 If $\gamma <  (4 N^3 C^4) ^{-1}$ then $\ps $ is a bi-Lipschitz homeomorphism, with 
 $(1-4 N^3 C^4 \gamma)^{-1}$  a Lipschitz constant of $\ps  ^{-1}$.
 \end{prop}

 \begin{proof}
 It suffices to show that for $z\notin \{a_1, \ldots, a_N\}$ and  for every unit vector $v\in \C$ 
 \begin{align}\label{prime}
 | (\ps (z) - z)'| \le  4 N^3 C^4 \gamma  ,
\end{align}
where by ``prime'' we denote any directional derivative $\frac {\partial } {\partial v} $, $v\in \C$, $|v|=1$.  
Indeed, if \eqref{prime} holds then clearly $\ps $ is Lipschitz.   Moreover,  if $\gamma < (4 N^3 C^4) ^{-1}$ 
then for any $p\in \C$, $z \to p+z-\ps (z)$ is a contraction and hence admits a unique fixed point $z_p$, that is a unique $z_p$ such that $\ps (z_p)=p$.  Hence $\ps$ is a homeomorphism by the invariance of domain.  By \eqref{prime} for any $p, q\in \C$,   $ | (\ps (p) - p)- (\ps (q) -q) | \le 4 N^3 C^4\gamma|p-q|$, that gives 
 \begin{align}
  | p- q | \le  (1- 4 N^3 C^4 \gamma)^{-1}  |\ps (p) - (\ps (q) |
\end{align}
if $\gamma < (4 N^3 C^4) ^{-1}$.   

To show \eqref{prime} we  use the following  bounds that follow from the conditions  (3)-(5).  
\begin{align}\label{boundonmu}  \notag
& |\mu_i(z)|\le C^{2-\frac 1 d}  |z-a_i|^{-1}  \mu (z) ^{1- \frac {1} d} ,  \\
& |\mu_i '(z)| \le N C^2  |z-a_i|^{-1}  \mu (z),  \\ \notag
&|\mu '(z)| \le N^2C^{2+ \frac 1 d} \mu (z)^{\frac {d+1} d}, 
\end{align}

By (3) we have $|\mu_i(z)|\le C|\xi_i|( max_j |\xi_j|)^{d-1})$, thus  by 5) we have 
$|\mu_i(z)|\le C^{1+\frac{d-1}{d}}|\xi_i| \mu (z) ^{\frac {d-1} {d}}$ i.e. the first inequality.
We present now a detailed proof of the second inequality.  By the chain rule 
$$
|\frac {\partial \mu_i} {\partial z} | + |\frac {\partial \mu_i} {\partial \bar z} | 
\le \sum_k (|\frac {\partial \f_i} { \partial  \xi_k}   \frac {\partial \xi_k} { \partial  z} |
+ |\frac {\partial \f_i} { \partial \bar \xi_k}   \frac {\partial\bar  \xi_k} { \partial \bar z} | ) 
= \sum_k (|\partial \f_i/ \partial \xi_k | + |\partial \f_i/ \partial \bar \xi_k |)  |\xi_k |^2 .
$$
Therefore by (4) and (5) 
$$
|\frac {\partial \mu_i} {\partial z} | + |\frac {\partial \mu_i} {\partial \bar z} | 
\le C \sum |\xi_i| ( max _j |\xi_j|)^{d} \le 
 C^2 N|\xi_i| \mu .
$$
Now for a unit vector $v=a +bi \in \C$, $a^2 + b^2=1$,  
$$
| a \frac {\partial \mu_i} {\partial x}  + i b \frac {\partial \mu_i} {\partial y} |  = 
| (a+ib) \frac {\partial \mu_i} {\partial z}  + (a-i b) \frac {\partial \mu_i} {\partial \bar z} | 
 \le 
 C^2 N|\xi_i| \mu .
$$
as required.  Using this inequality we get
$$|\mu'(z)|\le \sum_i |\mu_i '(z)|\le NC^2 \mu \sum_i |\xi_i|\le N^2C^2 \mu (max_j |\xi_j|)$$
 which by (5) gives
$|\mu'(z)|\le N^2C^{2+\frac{1}{d}}\mu^{1+\frac{1}{d}}$.

Given $z\in \C$, choose $j$ such that $|z-a_j| = \min _i |z-a_i|$.  Then, for all $i$,
\begin{align}\label{boundfordifference}
|a_i-a_j|\le 2 |z-a_i|.
\end{align}
By differentiating \eqref{psi-j}  
 \begin{align}\label{summands}
 | (\ps (z) - z)'| \le  \frac {\sum_{i\notin I_j} |\mu'_i (z) ( D_i- D_j)|} {\mu (z)} + 
  \frac {(\sum_{i\notin I_j} |\mu_i (z) ( D_i- D_j)|)  |\mu' (z)|} {( \mu (z))^2} .  
\end{align}
By \eqref{Dbound} and  \eqref{boundonmu} 
\begin{align*}
 |\mu'_i (z) ( D_i- D_j)| \le 2 N C^2 \gamma \mu (z) 
\end{align*}
and 
 \begin{align*} 
|\mu_i (z) ( D_i- D_j)|  |\mu' (z)| \le 2 N^ 2 C^4 \gamma( \mu (z))^2 .  
\end{align*}
This shows  \eqref{prime} and hence ends the proof of Proposition \ref{lipschitz}.  
\end{proof}

Consider $\ps$ as a function defined for $(z,a,b)\in \C\times \Sigma$, where 
$\Sigma=\{(a,b)\in \C^N \times \C ^N;  \text { such that  if } a_i = a_j \text { then } b_i=b_j\}$.  Thus 
\begin{align}\label{psi2} 
\ps (z,a,b) = \ps_{a,b} (z)  = z + \frac {\sum_{j=1}^N \mu_j (z,a) (b_j-a_j)} { \mu (z,a)} , 
\end{align}
where  $\mu_i(z,a) =f_i((z-a_1)^{-1} , \ldots, (z-a_N)^{-1} )$,  $\mu (z,a)  = \sum_i \mu_i(z,a)$, and 
 $\ps_{a,b} (a_i) = b_i$.  We may also consider  $\ps (z,a,b)$  as a family of functions  
 $\ps_{a,b}:\C\to \C$, parameterized by $a,b$.

\begin{prop}\label{continuity}
Let $a(x): X\to \C^N, b(x): X\to \C^N$ be continuous functions defined on a topological space $X$ such that 
for every $x\in X$ and $i, j$, if $a_i(x) = a_j(x)$ then $b_i(x) = b_j(x)$.  
Then $\ps (z,a(x),b(x))$ is continuous as a function of   $(x,z)$.  
\end{prop}

\begin{proof}
Let $(z,a,b) \to (z_0, a_0, b_0)$.    
Clearly $\ps (z,a,b)  \to  \ps (z_0,a_0,b_0)$ if $z_0 \not \in \{a_{01},\ldots, a_{0N}\}$.   Thus suppose $z_0=a_{01}$ and then $\ps (z_0,a_0,b_0)  = b_{01}$.  Let  $J=\{j\in \{1, . . . ,n\}; a_{0j}= a_{01}\}$.    Then 
\begin{align*}
 \ps (z,a,b)  - \ps (z_0,a_0,b_0)  & = (z - z_0)+  
\frac { \sum_{i\in J}  \mu_i (z,a) ((b_i -b_{01}) -(a_i-a_{01}))} { \mu (z,a) } \\
& + \frac { \sum_{i\notin J} \mu_i (z,a) ((b_{i}-b_{01}) -(a_{i}- a_{01}))} { \mu (z,a) }  . 
\end{align*}
We show that the last  two summands converge to $0$ as $(z,a,b) \to (z_0, a_0, b_0)$.  
Note that $b_{01}= b_{0i}, a_{01}= a_{0i}$ if $i\in J$.  Therefore 
\begin{align*}
& \frac { \sum_{i\in J}  \mu_i (z,a) ((b_i -b_{01}) -(a_i-a_{01}))} { \mu (z,a) }  
= \sum_{i\in J} \frac {  \mu_i (z,a) } { \mu (z,a) } ((b_i -b_{0i}) -(a_i-a_{0i})) . 
\end{align*} 
By (3) and (5) we always have that $|\frac{\mu_i}{\mu}|\le C^2,$ and using the fact that 
$((b_i -b_{0i}) -(a_i-a_{0i}))\to 0$ we get the second summand goes to zero.
 So  does the third one  because 
 $$
\frac { \mu_i (z,a) } { \mu (z,a) }  \to 0 
$$
if $i\notin J$.  To show this last property we note that $ \mu (z,a) \to \infty$,  the limit of $z-a_i$ is nonzero if $i\notin J$, and use the first inequality of  \eqref{boundonmu}.  
\end{proof}

\begin{rem}\label{gammalimit}
If $(a,b) \to (a_0,b_0)$ then $\gamma (a_0,b_0) \le \liminf \gamma (a,b),$ thus $\gamma$ is lower semi-continuous, where formally we put $\gamma (a,b) = 0$ if $a_1= \cdots =a_N$, $b_1= \cdots =b_N$. 
\end{rem}


\begin{example} \label{exampleinterpol1}
In the original Whitney interpolation $\f_i(\xi) = |\xi_i|$, cf. \cite{whitney}, see also  \cite{hardtsullivan}.     
   \end{example}
   
    \begin{example} \label{exampleinterpol2}
 In this paper we use the following family. 
For $\xi_1, \ldots ,\xi_N \in \C$  we denote by $\sigma_i = \sigma_i (\xi_1, \ldots ,\xi_N)$ 
the elementary symmetric functions  of $\xi_1, \ldots ,\xi_N$.    
Let  $P_k  =  \sigma^{\alpha_k}_k  $, where $\alpha_k =  (N!)/k$.   
Define 
\begin{align}\label{fi3}
\f_j (\xi) =  \frac 1 {N!}  \sum _k  \xi_j  \frac {\partial  P_k}{\partial \xi_j} \bar P_k.  
\end{align} 
and therefore it follows that 
\begin{align}\label{fi4}
\f (\xi) =\sum f_j(\xi)=  \sum _k  P_k(\xi)  \bar P_k (\xi) .  
\end{align}

Then $\ps$ equals   
\begin{align}\label{ourpsi} 
\ps (z,a,b) =  
z + \frac {\sum _k  \bigl ( \sum_{j=1}^N \ \xi_j  \frac {\partial  P_k}{\partial \xi_j}(\xi)(b_j-a_j)  \bigr ) \bar P_k (\xi) } 
{N!(  \sum _k  P_k \bar P_k (\xi)) } ,
\end{align}
where $\xi = ((z-a_1)^{-1} , \ldots, (z-a_N)^{-1})$.  
  \end{example}


  \bigskip
 \section{Generalized discriminants}\label{Part:discriminants}

We recall below the classical generalized discriminants, see e.g. \cite{whitneybook} Appendix IV.  
Let $\mathcal K$ be a field of characteristic zero and let 
\begin{align}\label{type}
F(Z) = Z^p + \sum_{j=1}^ p A_iZ^{p-i} = \prod_{j=1}^ p (Z-\xi_i)\in \mathcal K[Z], 
\end{align}
with the roots $\xi_i\in \overline {\mathcal K}$.    Then the expressions 
$$
D_{j} =  \sum_{r_1 <\cdots < r_{j}} \,  \prod_{k< l;\, k,l \in \{r_1, \ldots ,r_{j}\}}  (\xi_k-\xi_l)^2
$$
are symmetric in $\xi_1, \ldots , \xi_p$ and hence polynomials in $A_1, \ldots, A_{p}$.   
Thus $D_p$ is the standard discriminant and $F$ has exactly $d$ distinct roots if and only if 
$D_{d+1}=\cdots = D _{p}=0$ and $D_{d}\neq 0$.  The following lemma is obvious. 

\medskip

\begin{lem}\label{twodiscr}
Let $F\in \mathcal K[Z]$ be a monic polynomial of degree $p$ that has exactly $d$ distinct roots in 
$\xi_i\in \overline {\mathcal K}$ of multiplicities $\mathbf m=(m_1, ..., m_d)$.  Then there is a positive constant 
$C= C_{p,\mathbf m}$ such that  the generalized discriminant $D_{d,F}$ of $F$ and 
the standard discriminant $\Delta_{F_{red}}$ of $F_{red}$ are related by 
$$
D_{d,F}= C \Delta_{F_{red}}.
$$
\end{lem}

We often use the following consequence of the Implicit Function Theorem.

\begin{lem}\label{implicit}
Let $F\in \K \{x_1, ... ,x_n\}[Z]$ be a monic polynomial in $Z$ such that the discriminant 
$\Delta_{F_{red}}$ does not vanish at the origin.  Then, on a neighborhood $U$ of $0\in \K^n$, 
the complex roots $\xi_i(x_1, ... ,x_n)$ of $F$ are $\K$-analytic, distinct, and of constant multiplicities.  
\end{lem}


\end{document}